\documentclass[10pt, reqno]{amsart}
\usepackage[T1]{fontenc}

\usepackage[backend=biber, style=numeric, isbn=false, doi=false]{biblatex}
\usepackage{indentfirst}
\usepackage{amstext}
\usepackage{amsopn}
\usepackage{amsfonts}
\usepackage{amsmath}
\usepackage{latexsym}
\usepackage{amscd}
\usepackage{amssymb}
\usepackage{amsmath}
\usepackage[all,cmtip]{xy}
\usepackage{leftidx}
\usepackage{microtype}
\usepackage{colonequals} 
\usepackage{enumitem}
\usepackage{graphicx}
\usepackage[colorlinks=true,pagebackref=false,linktocpage=true]{hyperref}
\usepackage{tikz}
\usepackage{tikz-cd}
\usepackage[english]{babel}
\usepackage{csquotes}
\usepackage{adjustbox}

=\oddsidemargin \font\teneufm=eufm10 \font\seveneufm=eufm7
\font\fiveeufm=eufm5
\newfam\eufmfam
\textfont\eufmfam=\teneufm \scriptfont\eufmfam=\seveneufm
\scriptscriptfont\eufmfam=\fiveeufm

\newtheorem{stheorem}{Theorem}[]
\newtheorem{ques}[stheorem]{Question}
\newtheorem{prop}[stheorem]{Proposition}
\newtheorem{lemm}[stheorem]{Lemma}
\newtheorem{theo}[stheorem]{Theorem}

\newtheoremstyle{subsection-tweak}
   {11pt}
   {3pt}%
   {}
   {}%
   {\bfseries}
   {}%
   {.5em}
   {\thmnumber{\@{#1}{}\@{#2}.}%
    \thmnote{~{\bfseries#3.}}}    

\theoremstyle{subsection-tweak}
\newtheorem{pp}[stheorem]{}

\theoremstyle{definition}
\newtheorem{rema}[stheorem]{Remark}

\makeatletter
\renewcommand{\subsection}{\@startsection{subsection}{2}{\z@}%
  {-4.25ex \@plus -1.5ex \@minus -.2ex}
  {2.25ex \@plus .5ex}
  {\normalfont\large\itshape}}
\makeatother

\newcommand{\GG}{\mathbb{G}}
\newcommand{\FF}{\mathbb{F}}
\newcommand{\NN}{\mathbb{N}}
\newcommand{\ZZ}{\mathbb{Z}}
\newcommand{\QQ}{\mathbb{Q}}
\newcommand{\RR}{\mathbb{R}}
\newcommand{\CC}{\mathbb{C}}
\newcommand{\Ker}{\mathrm{ker}}
\newcommand{\coker}{\mathrm{coker}}
\newcommand{\Ima}{\mathrm{Im}}
\newcommand{\id}{\mathrm{id}}
\newcommand{\GL}{\mathrm{GL}}
\newcommand{\Gal}{\mathrm{Gal}}
\newcommand{\Br}{\mathrm{Br}}

\newcommand{\Oc}{\mathcal{O}}

\newcommand{\Lh}{\widehat{L}}
\newcommand{\Kh}{\widehat{K}}
\newcommand{\Khnr}{\Kh^\mathrm{unr}}
\newcommand{\Rh}{\widehat{R}}

\newcommand{\Galnr}{\Gamma^\mathrm{unr}}

\begin{document}

\title[Answer to a decomposition question on tori]{Answer to a decomposition question on tori raised by Colliot-Thélène and Sansuc}

\author[A.\ Zidani]{Anis Zidani}
\address{Institut de Mathématiques de Jussieu-Paris Rive Gauche, 
Sorbonne Université, Campus Pierre et Marie Curie,
4, place Jussieu, 75252 Paris Cedex 05, France}

\date{\today}

\begin{abstract} 
    The aim of this note is to present a simple strategy to answer negatively a decomposition question on tori posed by Colliot-Thélène and Sansuc in the article \textit{Principal Homogeneous Spaces under Flasque Tori: Applications} of 1987. We then deduce a torus $T$ over $\QQ$ and a prime number $p$ such that $T(\ZZ_p)\,T(\QQ)\not=T(\QQ_p)$, where $T(\ZZ_p)$ denotes the maximal compact subgroup of $T(\QQ_p)$.
\end{abstract}

\maketitle


\medskip


\bigskip

\tableofcontents

\section*{Introduction}

The purpose of this note is to answer the following question:

\begin{ques}[{\cite[Remark 8.3]{CTS}}]\label{QuestionCT}
    Let $K$ be a global field, $\Kh$ a completion of $K$ with respect to a discrete valuation, and $\Rh$ the associated discrete valuation ring. Consider a torus $T$ over $K$. Let $T(\Rh)$ denote the maximal compact subgroup of $T(\Kh)$. Is the equality
    $$T(\Rh)\, T(K) = T(\Kh)$$
    satisfied?
\end{ques}

\begin{pp}[{Context}]
Question~\ref{QuestionCT} was first asked in 1987 by Bruhat and Tits to Colliot-Thélène and Sansuc. In that same paper, they prove the case where $T$ is an unramified $K$-torus using the theory of flasque resolutions (see~\cite[Proposition 8.1.(i)]{CTS}). 

Later, in 2007 in \cite{CTSuresh}, Colliot-Thélène and Suresh studied variants of the problem: the case where $R$ is a semi-local ring (thus involving several valuations to be considered simultaneously), and the case of a decomposition ``$RT(\Kh)\,T(\Rh) = T(\Kh)$'', where $RT(\Kh)$ is the subgroup of $R$-trivial elements of $T(\Kh)$ (see~\cite[7.1.~Définition]{BourbakiGille}). They succeeded in finding counterexamples for both situations. They also provided a list of cases where the decomposition ``$RT(\Kh)\,T(\Rh) = T(\Kh)$'' holds (see~\cite[Proposition 2.1]{CTSuresh}). However, even after this work, Question~\ref{QuestionCT} remained open. 
\end{pp}

In this note, we provide a negative answer to Question~\ref{QuestionCT} (see~Propositions~\ref{ContreExempleCorpsDeNombre} and \ref{ContreExempleCorpsDeFonction}). More generally, we develop a systematic strategy to construct tori that do not satisfy the required decomposition. The crucial ingredient for this is the Kottwitz morphism, which had not been used until now for this type of problem.

\begin{pp}[{The Kottwitz morphism}]\label{paragrapheKottiwtz}
Suppose this time that $\Kh$ denotes a discretely valued Henselian field with valuation $v$ and whose residue field is perfect. The Kottwitz morphism can be viewed as a functorial generalization of the valuation morphism
$$\begin{aligned}
    \kappa_{\GG_m} \colon \Kh^\times & \longrightarrow \ZZ \\
     k & \longmapsto v(k)
\end{aligned}$$

to any reductive group $G$ over $\Kh$. It was defined by Kottwitz in \cite[Section 7]{Kottwitz}.
\medskip

For example, the Kottwitz morphism for $\GL_n$ over $\Kh$ is:
$$\begin{aligned}
    \kappa_{\GL_n} \colon \GL_n(\Kh) & \longrightarrow \ZZ \\
     M & \longmapsto v(\mathrm{det}(M)),
\end{aligned}$$

while for $\mathrm{PGL}_n$ over $\Kh$, we have:
$$\begin{aligned}
    \kappa_{\mathrm{PGL}_n}\colon \mathrm{PGL}_n(\Kh) & \longrightarrow \ZZ / n \ZZ  \\
     \overline{M} & \longmapsto v(\mathrm{det}(\overline{M})) \pmod n .
\end{aligned}$$

The morphism is defined using the algebraic fundamental group (see~\cite[Definition 11.3.2]{kaletha_prasad_2023}), denoted by $\pi_1(G)$. It is a $\Gal(\Kh^\mathrm{sep}/\Kh)$-module (where $\Kh^\mathrm{sep}$ is the separable closure of $\Kh$). \linebreak For a $\Kh$-torus $T$, we have $\pi_1(T)=T_*$, the Galois module of cocharacters of $T$.
\medskip

Consider $\Khnr$, the maximal unramified extension of $\Kh$. Let $\Galnr\colonequals \Gal( \Khnr/\Kh)$ and ${I\colonequals \Gal(\Kh^\mathrm{sep}/\Khnr)}$. Consider a reductive group $G$ over $\Kh$. It turns out that the Kottwitz morphism of $G$ over $\Khnr$ is a morphism of $\Galnr$-modules: $\kappa_{G_{\Khnr}}:G(\Khnr)\rightarrow \pi_1(G)_I$.
\medskip

The Kottwitz morphism of $G$ over $\Kh$ is then the morphism $\kappa_G:G(\Kh)\rightarrow (\pi_1(G)_I)^{\Galnr}$ obtained from $\kappa_{G_{\Khnr}}$ by taking the $\Galnr$-invariants. It enjoys many properties:

\begin{itemize}
    \item It is surjective when $\Kh$ is a local field (see~\cite[Corollary 11.7.6]{kaletha_prasad_2023}).
    \item It is trivial for any simply connected semisimple group (since its $\pi_1$ is trivial).
    \item Its kernel is $G(\Kh)^0$, the open subgroup generated by the parahoric subgroups of $G$ (see~\cite[Proposition 11.5.4]{kaletha_prasad_2023}).
    \item The preimage of the torsion part of $(\pi_1(G)_I)^{\Galnr}$ is $G(\Kh)^1$, the intersection of the kernels of $g\mapsto v(\chi(g))$ for $\chi$ running through the characters of $G$ (see~\cite[Lemma 11.5.2]{kaletha_prasad_2023}).
\end{itemize}

Furthermore, for a $\Kh$-torus $T$, it turns out that $T(\Kh)^0$ is the subgroup of points of the connected Néron model of $T$ (see~\cite[Corollary B.8.7]{kaletha_prasad_2023}), and $T(\Kh)^1$ is the maximal bounded subgroup of $T(\Kh)$.
\medskip
\end{pp}

\begin{pp}[{Strategy and outline of the paper}]
Our strategy to answer Question~\ref{QuestionCT} consists of two steps.
In Section~\ref{SectionLemmesKottwitz}, we prove that for a reductive group $G$ over $\Kh$ (which is now a local field) split by an extension whose wild inertia subgroup is cyclic, the map ${G(\Kh)/\overline{G(K)}\rightarrow G(\Kh)/G(K)\,G(\Kh)^0}$ is an isomorphism (see Proposition~\ref{ThmCorpsNbr}). This is proven using the Kottwitz morphism. Next, in Section~\ref{SectionCEX}, we find a torus $T$ over $K$ satisfying the previous hypothesis (in fact, split by a tamely ramified extension) such that ${T(\Kh)^1=T(\Kh)^0}$ and ${\overline{T(K)}\not = T(\Kh)}$ (meaning that $T$ does not satisfy weak approximation). Weak approximation over number fields is very well understood thanks to the work of Sansuc (see~\cite{Sansuc}). In \cite{borovoi2025defectweakapproximationreductive}, Borovoi provides a version including function fields.
We then deduce in Propositions~\ref{ContreExempleCorpsDeNombre} and \ref{ContreExempleCorpsDeFonction} a negative answer to Question~\ref{QuestionCT} (and even a negative answer to an analogous question in the case of a semisimple group). It is worth noting that this strategy also works for studying the semi-local variant of the problem.

In addition, we discuss some further problems surrounding the Kottwitz morphism in Sections~\ref{SectionCEXSurjectivite} and \ref{SectionBirat} of this article.
On the one hand, in Section~\ref{SectionCEXSurjectivite}, we give two examples of tori over $\RR(\!(t)\!)$ whose Kottwitz morphisms are not surjective. Thus, surjectivity is only achieved for specific fields. On the other hand, we show in Section~\ref{SectionBirat} that, for a $K$-torus $T$, the groups ${\coker(\kappa_{T_{\Kh}})}$, ${\coker(T(K)\rightarrow ((T_*)_I)^{\Galnr})}$, and ${T(\Kh)/T(K)\, T(\Kh)^0}$ are stably birational invariants of $T$.
\end{pp}

\section*{Notations}

Let $K$ be a field. We denote:
\begin{itemize}
    \item $\Omega_K$ the set of places of $K$, $\Sigma$ a finite subset of $\Omega_K$, and $\overline{\Sigma}\colonequals \Omega_K \backslash \Sigma$,
    \item $K_v$ the completion of $K$ at a place $v\in \Sigma$ and $K_\Sigma\colonequals  \prod_{v\in \Sigma}K_v$ (with $K$ diagonally embedded),
    \item $K_v^{\mathrm{unr}}$ the maximal unramified extension of $K_v$ and $K_\Sigma^\mathrm{unr}\colonequals \prod_v K_v^\mathrm{unr}$,
    \item $\Galnr_v\colonequals  \Gal(K_v^\mathrm{unr}/K_v)$, $I_v$ the inertia subgroup of $K_v$, $\Galnr\colonequals \prod_v \Galnr_v$, and $I\colonequals \prod_v I_v$.
\end{itemize}
We assume that, for every $v\in \Sigma$, the residue field of $K_v$ is \emph{perfect}. The notations, properties, and definitions introduced in this note for a complete field generalize to the case of $K_\Sigma$ by working factor by factor. This is implicit throughout the document.

We also adopt the notations $(-)^\mathrm{sep}$, $\pi_1(-)$, $(-)_*$, $\kappa_{(-)}$, $(-)^0$, and $(-)^1$ from $\S$\ref{paragrapheKottiwtz}.
\medskip

In Sections~\ref{SectionLemmesKottwitz} and \ref{SectionCEX}, the field $K$ is assumed to be \emph{global}.

\section{A few lemmas on the Kottwitz morphism}\label{SectionLemmesKottwitz}

The purpose of this section is to establish several results crucial for developing our answer to Question~\ref{QuestionCT}, which involve the Kottwitz morphism. In particular, we show that under certain conditions, the problem reduces to a question of weak approximation.

\begin{pp}[{Cyclic wild inertia}] 
A reductive group over $K_\Sigma$ is said to have \emph{cyclic wild inertia} if for every $v\in \Sigma$, there exists a Galois splitting extension of $K_v$ such that its wild inertia subgroup is cyclic (this is in particular the case if the extension is tamely ramified). We observe that this is equivalent to requiring the inertia subgroup to be a metacyclic group (that is, one whose Sylow subgroups are cyclic).
\end{pp}

\begin{lemm}\label{FactDirectTore}
    Let $T$ be a torus over $K_\Sigma$ which is a direct factor of a quasi-trivial torus over $K_\Sigma^\mathrm{unr}$. \linebreak Then $\pi_1(T)_I$ is free and $T(K_\Sigma)^0=T(K_\Sigma)^1$.
\end{lemm}

\begin{proof}
    The construction of the functor $\pi_1(-)_I$ behaves well with respect to direct products, so that $\pi_1(T)_I$ is a direct factor of the ``$\pi_1(-)_I$'' of a quasi-trivial torus, which is free. Consequently, $\pi_1(T)_I$ is free. By \cite[Corollary 11.1.6]{kaletha_prasad_2023}, we then have $T(K_\Sigma)^0=T(K_\Sigma)^1$.
\end{proof}

\begin{lemm}\label{FlasqueTame}
    Let $T$ be a $K_\Sigma$-torus with cyclic wild inertia that is either flasque or coflasque. Then $T$ is a direct factor of a quasi-trivial torus over $K_\Sigma^\mathrm{unr}$, the group $\pi_1(T)_I$ is free, and $T(K_\Sigma)^0=T(K_\Sigma)^1$.
\end{lemm}

\begin{proof}
    The theorem of Endo and Miyata (see~\cite[0.5]{CTS}) shows that a flasque or coflasque torus is a direct factor of a quasi-trivial torus when this torus is split by a metacyclic extension. Therefore, $T$ is a direct factor of a quasi-trivial torus over $K_\Sigma^\mathrm{unr}$. The rest then follows from Lemma~\ref{FactDirectTore}.
\end{proof}

\pagebreak

\begin{lemm}\label{SuiteKottwitz}
    Let $1\rightarrow A \rightarrow B \rightarrow C \rightarrow 1$ be an exact sequence of reductive $K_\Sigma$-groups. Assume that $A$ is a torus and that the group $\pi_1(A)_I$ is free. Then we have the following exact sequences:
    \[\begin{tikzcd}
        1 & {\pi_1(A)_I} & {\pi_1(B)_I} & {\pi_1(C)_I} & 1
        \arrow[from=1-1, to=1-2]
        \arrow[from=1-4, to=1-5]
        \arrow[from=1-2, to=1-3]
        \arrow[from=1-3, to=1-4]
    \end{tikzcd}\]
    and
    \[\begin{tikzcd}
        1 & {A(K_\Sigma)^0} & {B(K_\Sigma)^0} & {C(K_\Sigma)^0} & 1.
        \arrow[from=1-1, to=1-2]
        \arrow[from=1-4, to=1-5]
        \arrow[from=1-2, to=1-3]
        \arrow[from=1-3, to=1-4]
    \end{tikzcd}\]
\end{lemm}

\begin{proof}
    Recall that the functor $\pi_1(-)$ is exact (see~\cite[Theorem 3.8]{borovoi2013algebraicfundamentalgroupreductive}). Taking coinvariants under $I$, we obtain:
    \[\begin{tikzcd}
        H_1(I,\pi_1(A)) & {\pi_1(A)_I} & {\pi_1(B)_I} & {\pi_1(C)_I} & 1.
        \arrow[from=1-1, to=1-2]
        \arrow[from=1-4, to=1-5]
        \arrow[from=1-2, to=1-3]
        \arrow[from=1-3, to=1-4]
    \end{tikzcd}\]
    Since $\pi_1(A)_I$ is free by Lemma~\ref{FactDirectTore}, the image of $H_1(I,\pi_1(A))$ (which is a torsion group) in $\pi_1(A)_I$ is trivial. Hence the first exact sequence.
    
    As for the second exact sequence, since the residue field is perfect, Steinberg's theorem (see~\cite[Theorem 2.3.3.(1)]{kaletha_prasad_2023}) gives $H^1(I,A(K_\Sigma^\mathrm{sep}))=1$. We then have:
    \[\begin{tikzcd}
        1 & {A(K_\Sigma^\mathrm{unr})} & {B(K_\Sigma^\mathrm{unr})} & {C(K_\Sigma^\mathrm{unr})} & H^1(I,A(K_\Sigma^\mathrm{sep})) = 1.
        \arrow[from=1-1, to=1-2]
        \arrow[from=1-4, to=1-5]
        \arrow[from=1-2, to=1-3]
        \arrow[from=1-3, to=1-4]
    \end{tikzcd}\]
    Combining the two previous exact sequences via the Kottwitz morphism and using the Snake Lemma, we obtain:
    \[\begin{tikzcd}
        1 & {A(K_\Sigma^\mathrm{unr})^0} & {B(K_\Sigma^\mathrm{unr})^0} & {C(K_\Sigma^\mathrm{unr})^0} &  1.
        \arrow[from=1-1, to=1-2]
        \arrow[from=1-4, to=1-5]
        \arrow[from=1-2, to=1-3]
        \arrow[from=1-3, to=1-4]
    \end{tikzcd}\]
    Taking invariants under $\Galnr$ and using the fact that $H^1(\Galnr,A(K_\Sigma^\mathrm{unr})^0)=1$ by \cite[Lemma 11.7.1]{kaletha_prasad_2023}, we then obtain:
    \[\begin{tikzcd}
        1 & {A(K_\Sigma)^0} & {B(K_\Sigma)^0} & {C(K_\Sigma)^0} &  H^1(\Galnr,A(K_\Sigma^\mathrm{unr})^0)=1.
        \arrow[from=1-1, to=1-2]
        \arrow[from=1-4, to=1-5]
        \arrow[from=1-2, to=1-3]
        \arrow[from=1-3, to=1-4]
    \end{tikzcd}\]
    Hence the desired exact sequence.
\end{proof}

\begin{prop}\label{ThmCorpsNbr}
    Let $G$ be a reductive group over $K$. If the group $G_{K_\Sigma}$ has cyclic wild inertia, then the following natural map is an isomorphism:
    $$G(K_\Sigma)/\overline{G(K)}\overset{\sim}{\rightarrow}G(K_\Sigma)/G(K)\,G(K_\Sigma)^0.$$
\end{prop}

\begin{proof}
    Since $G(K_\Sigma)^0$ is an open subgroup of $G(K_\Sigma)$, we have the inclusion ${\overline{G(K)}\subset G(K)G(K_\Sigma)^0}$. It is therefore sufficient to prove that $G(K_\Sigma)^0 \subset \overline{G(K)}$. Note that since the group $G_{K_\Sigma}$ has cyclic wild inertia, we must have a flasque resolution (see~\cite[Proposition-Définition 3.1]{FlasqueGrpLin}) over $K_\Sigma^\mathrm{unr}$ given by ${1 \rightarrow F' \rightarrow H' \rightarrow G_{K_\Sigma^\mathrm{unr}}\rightarrow 1}$ where $F'$ also has cyclic wild inertia. Consequently, it is a direct factor of a quasi-trivial torus by Lemma~\ref{FlasqueTame}.
    
    Now consider a flasque resolution of $G$ over $K$ given by $1\rightarrow F \rightarrow H \rightarrow G \rightarrow 1$. Since $F$ is unique up to a quasi-trivial torus (see~\cite[Proposition 3.2.(iii)]{FlasqueGrpLin}), the previous discussion shows that $F$ is in fact a direct factor of a quasi-trivial torus over $K_\Sigma^\mathrm{unr}$. Furthermore, \cite[Proposition 2.6]{borovoi2025defectweakapproximationreductive} indicates that $H$ satisfies weak approximation, in other words: $\overline{H(K)}=H(K_\Sigma)$. This implies that the subgroup $H(K_\Sigma)^0\subset H(K_\Sigma)$ maps into $\overline{G(K)}$. But Lemma~\ref{SuiteKottwitz} tells us that $H(K_\Sigma)^0\rightarrow G(K_\Sigma)^0$ is surjective, so that $G(K_\Sigma)^0\subset \overline{G(K)}$ as desired.
\end{proof}

\begin{rema}
    The result no longer holds if the group does not have cyclic wild inertia. Consider the extension $L/K\colonequals \QQ(\zeta_8)/\QQ=\QQ(\sqrt{-1},\sqrt{2})/\QQ$ (where $\zeta_8=e^{\frac{2i\pi}{8}}$) and consider the norm torus $T\colonequals R^{(1)}_{L/K}(\GG_m)$. According to \cite[11.6.~Example 3.2)]{VoskresenskiiBook}, $\QQ_2(\zeta_8)/\QQ_2$ is a totally ramified biquadratic extension such that $T(\QQ_2)/\overline{T(\QQ)} = \ZZ/2\ZZ \not = 0$. However, $\coker(T(\QQ)\rightarrow \Ima(\kappa_{T_{\QQ_2}}))=0$ \linebreak (i.e., $T(\QQ_2)/T(\QQ)\, T(\QQ_2)^0 = 0$). Indeed, by adapting Section~\ref{Section1erCEX} to our context and using its notation, we observe that $N(a(1-\zeta_8)) \in T(\QQ)$ and $N(b(1-\zeta_8)) \in T(\QQ)$ map respectively to $\overline{a}$ and $\overline{b}$ under $\kappa_{T_{\QQ_2}}$, and thus generate $\Ima(\kappa_{T_{\QQ_2}})$. 
\end{rema}

\section{Weak approximation and construction of examples}\label{SectionCEX}

In this section, we compute defects of weak approximation. Combined with Proposition~\ref{ThmCorpsNbr}, we provide a negative answer to Question~\ref{QuestionCT} (and to its semisimple variant), when $K$ is a number field (see Proposition~\ref{ContreExempleCorpsDeNombre}) and when $K$ is a function field (see Proposition~\ref{ContreExempleCorpsDeFonction}).

\begin{lemm}\label{CalculCokerMu}
    Let $\mu\colonequals R_{K'/K}(\mu_n)/\mu_n$ with $n\in \NN^\times$ such that $\mathrm{char}(K) \nmid n$. We have:
    $$\coker \left(H^1(K,\mu)\rightarrow H^1(K_\Sigma,\mu)\right)\overset{\sim}{\rightarrow} \coker\left(\Br(K'/K)[n]\rightarrow \Br(K'_{\Sigma'}/K_\Sigma)[n]\right) \overset{\sim}{\rightarrow} \ZZ/\frac{\gcd(d_\Sigma,n)}{\gcd(d_\Sigma,d_{\overline{\Sigma}},n)} \ZZ,$$
    where for a set of places $S$ of $K$, we denote $d_S\colonequals \mathrm{lcm}_{v\in S}[K'_{v'}:K_v]$, and where for every $v\in \Omega_K$, \linebreak $v'$ denotes a place of $K'$ lying over $v$. 
\end{lemm}

\begin{proof}
    Let $T\colonequals  R_{K'/K}(\GG_m)/\GG_m$. Observe that $\mu=\ker(T\overset{\times n}{\rightarrow} T)$. Using the long exact sequence in cohomology over $K$ and $K_\Sigma$, we obtain the following diagram with exact rows:
    \[\begin{tikzcd}
        1 & {T(K)/nT(K)} & {H^1(K,\mu)} & {\Br(K'/K)[n]} & 1 \\
        1 & {T(K_\Sigma)/nT(K_\Sigma)} & {H^1(K_\Sigma,\mu)} & 
    {\Br(K'_{\Sigma'}/K_\Sigma)[n]} & 1.
        \arrow[from=1-1, to=1-2]
        \arrow[from=1-2, to=1-3]
        \arrow[from=1-2, to=2-2]
        \arrow[from=1-3, to=1-4]
        \arrow[from=1-3, to=2-3]
        \arrow[from=1-4, to=1-5]
        \arrow[from=1-4, to=2-4]
        \arrow[from=2-1, to=2-2]
        \arrow[from=2-2, to=2-3]
        \arrow[from=2-3, to=2-4]
        \arrow[from=2-4, to=2-5]
    \end{tikzcd}\]

    Since $T$ admits by definition a flasque resolution ${1\rightarrow \GG_m \rightarrow R_{K'/K}(\GG_m)\rightarrow T \rightarrow 1}$, it is retract \mbox{$K$-rational} (see~\cite[Proposition 7.4]{CTS}). Consequently, $T(K)$ is dense in $T(K_\Sigma)$ (see~\cite[Proposition 8.1.(iii)]{CTS}). Furthermore, since $\mathrm{char}(K)\nmid n$, multiplication by $n$ in $T$ is smooth, so that $nT(K_\Sigma)$ is an open subset of $T(K_\Sigma)$ (see~\cite[Lemma 3.1.2]{GGMB}). We therefore have ${T(K) \cdot nT(K_\Sigma) = T(K_\Sigma)}$. This implies that $T(K)/nT(K)\rightarrow T(K_\Sigma)/nT(K_\Sigma) $ is surjective.
    \medskip

The Snake Lemma applied to the previous diagram then yields the isomorphism:
    $$\coker \left(H^1(K,\mu)\rightarrow H^1(K_\Sigma,\mu)\right)\overset{\sim}{\rightarrow}\coker\left(\Br(K'/K)[n]\rightarrow \Br(K'_{\Sigma'}/K_\Sigma)[n]\right).$$
    Let us now compute the right-hand side. The Brauer-Hasse-Noether theorem (see~\cite[Theorem 14.11]{HarariCorpsDeClasses}) restricted to $\Br(K'/K)[n]$ yields the following exact sequence:
    \[\begin{tikzcd}
        1 & {\Br(K'/K)[n]} & {\bigoplus_{v\in \Omega_K} \Br(K'\otimes_K K_v /K_v)[n]} 
    & \frac{1}{n}\ZZ/\ZZ. 
        \arrow[from=1-1, to=1-2]
        \arrow[from=1-2, to=1-3]
        \arrow["{\sum \mathrm{inv}_v}", from=1-3, to=1-4]
    \end{tikzcd}\]
    Consider the partial maps $\phi_\Sigma\colonequals \sum_{v\in \Sigma} \mathrm{inv}_v$ and $\phi_{\overline{\Sigma}}\colonequals \sum_{v\in \overline{\Sigma}} \mathrm{inv}_v$. We thus have ${\phi_\Sigma(x)=-\phi_{\overline{\Sigma}}(x)=\phi_{\overline{\Sigma}}(-x)}$ for $x\in \Br(K'/K)[n]$. In fact, $\Ima((\phi_\Sigma)_{\mid\Br(K'/K)[n]})=\Ima(\phi_\Sigma)\cap \Ima(\phi_{\overline{\Sigma}})$. This finally allows us to deduce the following exact sequence:
    \[\begin{tikzcd}
        1 & {\Br(K'/K)[n]} & {\Br(K'_{\Sigma'}/K_\Sigma)[n]} & 
    {\Ima(\phi_\Sigma)/(\Ima(\phi_\Sigma)\cap \Ima(\phi_{\overline{\Sigma}}))}
     & 1.
        \arrow[from=1-1, to=1-2]
        \arrow[from=1-2, to=1-3]
        \arrow["{\phi_\Sigma}", from=1-3, to=1-4]
        \arrow[from=1-4, to=1-5]
    \end{tikzcd}\]
    To conclude, we compute the order of the quotient (which is cyclic). Let $v$ be a place of $K$. Observe that $\Br(K'_{v'}/K_v)=\Br(K'\otimes_K K_v /K_v)$. Also, $d_v\colonequals [K'_{v'}:K_v]=| \Br(K'_{v'}/K_v) |$ and the image of $\Br(K'_{v'} /K_v)$ in $\QQ/\ZZ$ is $(\frac{1}{d_v}\ZZ)/\ZZ$ (see~\cite[Corollary 8.10]{HarariCorpsDeClasses}). Thus, the image of $\Br(K'\otimes_K K_v /K_v)[n]$ is $(\frac{1}{\gcd(d_v,n)}\ZZ)/\ZZ$. Consequently, $\Ima(\phi_\Sigma)=\frac{1}{\gcd(d_\Sigma,n)}\ZZ/\ZZ$ and $\Ima(\phi_{\overline{\Sigma}})=\frac{1}{\gcd(d_{\overline{\Sigma}},n)}\ZZ/\ZZ$. We then observe that $\Ima(\phi_{\overline{\Sigma}})\cap \Ima(\phi_\Sigma) = \frac{1}{\gcd(|\Ima(\phi_\Sigma)|, |\Ima(\phi_{\overline{\Sigma}})|)}\ZZ / \ZZ=\frac{1}{\gcd(d_\Sigma,d_{\overline{\Sigma}},n)}\ZZ/\ZZ$.
    We finally conclude that the cokernel has order $\frac{\gcd(d_\Sigma,n)}{\gcd(d_\Sigma,d_{\overline{\Sigma}},n)}$ as desired.
\end{proof}

\begin{prop}\label{stratégie}
    Let $K'/K$ be a finite Galois extension, and let $\mu$ be a finite étale $K$-group of multiplicative type split over $K'$. Set $\Sigma'\colonequals \left\{w\in \Omega_{K'} \mid \exists v\in \Sigma, w\mid v\right\}$. We then have $K'_{\Sigma'}=K'\otimes_K K_\Sigma$. Assume that the wild inertia subgroup of $K'_{\Sigma'}/K_\Sigma$ is cyclic.
    Then, there exist a semisimple $K$-group $G$ and a $K$-torus $T$ with $T(K_\Sigma)^0=T(K_\Sigma)^1$, both split over $K'$, such that:
    $$G(K_\Sigma)/G(K)\,G(K_\Sigma)^0 \overset{\sim}{\leftarrow} G(K_\Sigma)/\overline{G(K)}\overset{\sim}{\rightarrow}\coker \left(H^1(K,\mu)\rightarrow H^1(K_\Sigma,\mu)\right) \text{\hspace{2pt} and }$$
    $$T(K_\Sigma)/T(K)\,T(K_\Sigma)^0 \overset{\sim}{\leftarrow}T(K_\Sigma)/\overline{T(K)}\overset{\sim}{\rightarrow}\coker \left(H^1(K,\mu)\rightarrow H^1(K_\Sigma,\mu)\right).$$
\end{prop}

\begin{proof}
    The left-hand isomorphisms are a consequence of Proposition~\ref{ThmCorpsNbr}. Indeed, $G$ and $T$ have cyclic wild inertia since they are split over $K'$, and thus over $K'_{\Sigma'}$.

    There exists a finite family of integers $(n_i)$ such that $\mu_{K'}\cong \prod \mu_{n_i}$. We then embed $\mu$ into the center of $G'\colonequals \mathrm{R}_{K'/K}(\prod \mathrm{SL}_{n_i})$. Setting $G\colonequals  G'/\mu$, we obtain a central isogeny with kernel $\mu$, where $G$ is a semisimple group (this was already known to Serre, see~\cite[p.158]{CohoGalois}). On the other hand, according to \cite[Proposition 1.3]{CTS}, one can find an exact sequence ${1\rightarrow \mu \rightarrow P \rightarrow T \rightarrow 1}$ such that $P$ is quasi-trivial and $T$ is coflasque. We can also observe that, by construction, $P$ and $T$ can be chosen to be split over $K'$. \cite[Theorem 2.12]{borovoi2025defectweakapproximationreductive} then gives the right-hand isomorphisms.

    Finally, Lemma~\ref{FlasqueTame} shows that $T(K_\Sigma)^0=T(K_\Sigma)^1$.
\end{proof}

\subsection{An example in the case of number fields}

\begin{pp}[{Study of a number field}]\label{LeCorpsDeNombres}
Let $K\colonequals \QQ$, $K'\colonequals \QQ(\sqrt{3},\sqrt{11})$, and $\Sigma\colonequals \{3\}$. We thus have $K_\Sigma=\QQ_3$. As for $K'_{\Sigma'}$, we have:
$$K'_{\Sigma'}\colonequals K'\otimes_\QQ \QQ_3 \cong \QQ[X,Y]/(X^2-3,Y^2-11) \otimes_\QQ \QQ_3 \cong \QQ_3[X,Y]/(X^2-3,Y^2-11).$$ 
Furthermore, $X^2-3$ is irreducible over $\QQ_3$ since $3$ is a uniformizer. Moreover, $X^2-11$ has no root in $\ZZ/3\ZZ$, and therefore none in $\ZZ_3(\sqrt{3})$. Since $X^2-11\in \ZZ_3(\sqrt{3})[X]$, it also has no root in $\QQ_3(\sqrt{3})$. Consequently, $\QQ_3[X,Y]/(X^2-3,Y^2-11)=\QQ_3(\sqrt{3},\sqrt{11})$ is a field. As a result, $\Sigma'=\{\sqrt{3}\}$ and $K'_{\Sigma'}\cong \QQ_3(\sqrt{3},\sqrt{11})$, which is of dimension $4$ and, of course, has Galois group $\ZZ/2\ZZ\times \ZZ/2\ZZ$. Therefore, $K'_{\Sigma'}/K_\Sigma$ is tamely ramified.
\medskip
    
Let us find the prime numbers for which $\QQ(\sqrt{3},\sqrt{11})/\QQ$ is ramified. Of course, all decomposition groups are cyclic in the unramified case. According to~\cite[Chapter 2, Exercise 42.(f)]{Marcus}, the discriminant of this extension is $16\times 3^2 \times 11^2$. The prime numbers dividing the discriminant, that is, the ramified primes, are then $2$, $3$, and $11$.

\begin{itemize}[leftmargin=0.5cm]
    \item At $p=2$, we note that $P\colonequals X^2-33$ has $1$ as a root modulo $8$ (and ${P'(1)=2 \pmod {4}}$). This means that $|P(1)|_2 \leq \frac{1}{8} < |P'(1)|_2^2 =\left(\frac{1}{2}\right)^2$. Newton's method then yields a root in $\QQ_2$. We also note that $X^2-3$ has no root in $\ZZ/4\ZZ$, hence no root in $\ZZ_2$, and consequently no root in $\QQ_2$ since $X^2-3\in \ZZ_2[X]$.
    Thus, $X^2-3$ is irreducible over $\QQ_2$. We then obtain:
    $$K'\otimes_\QQ \QQ_2 \cong \QQ_2[X,Y]/(X^2-3,Y^2-11) \cong \left(\QQ_2[\sqrt{3}]\right)[Y]/(Y^2-(\frac{\sqrt{33}}{\sqrt{3}})^2) \cong \left(\QQ_2(\sqrt{3})\right)^2.$$
    The two local fields of $K'$ lying over $2$ are therefore cyclic extensions of degree $2$. 
    \medskip

    \item At $p=11$, $P\colonequals X^2-3$ has $5$ as a root modulo $11$ (and $P'(5)\not=0 \pmod {11}$). By Hensel's lemma, $P\colonequals X^2-3$ has a root in $\QQ_{11}$. Furthermore, $X^2-11$ is irreducible over $\QQ_{11}$ since $11$ is a uniformizer. We then have:
    $$K'\otimes_\QQ \QQ_{11} \cong \QQ_{11}[X,Y]/(X^2-3,Y^2-11) \cong \left(\QQ_{11}(\sqrt{11})\right)[X]/(X^2-(\sqrt{3})^2) \cong \left(\QQ_{11}(\sqrt{11})\right)^2.$$
    Thus, the two local fields of $K'$ lying over $11$ are cyclic extensions of degree $2$.
\end{itemize}
\end{pp}

\begin{prop}\label{ContreExempleCorpsDeNombre}
    There exist a semisimple $\QQ$-group $G$ and a $\QQ$-torus $T$ with $T(\QQ_3)^0=T(\QQ_3)^1$, both split over $\QQ(\sqrt{3},\sqrt{11})$, such that:
    $$0\not = \ZZ/2\ZZ \overset{\sim}{\leftarrow} G(\QQ_3)/\overline{G(\QQ)}\overset{\sim}{\rightarrow}G(\QQ_3)/G(\QQ)\,G(\QQ_3)^0 \text{\hspace{2pt} and }$$
    $$0\not = \ZZ/2\ZZ \overset{\sim}{\leftarrow} T(\QQ_3)/\overline{T(\QQ)}\overset{\sim}{\rightarrow}T(\QQ_3)/T(\QQ)\,T(\QQ_3)^0.$$
\end{prop}

\pagebreak

\begin{proof}
    Let us resume the context of $\S$\ref{LeCorpsDeNombres}. Under the notation of Lemma~\ref{CalculCokerMu}, we have $d_\Sigma=4$ and $d_{\overline{\Sigma}}=2$. Set $\mu\colonequals R_{K'/K}(\mu_4)/\mu_4$. This same lemma then yields
    $$\coker \left(H^1(K,\mu)\rightarrow H^1(K_\Sigma,\mu)\right)\overset{\sim}{\rightarrow} \ZZ/\frac{\gcd(d_\Sigma,4)}{\gcd(d_\Sigma,d_{\overline{\Sigma}},4)} \ZZ = \ZZ/2\ZZ.$$
    
    The result then follows from Proposition~\ref{stratégie}.
\end{proof}

\subsection{An example in the case of function fields}

\begin{pp}[{Study of a function field}]
Let $p$ be a prime such that $p\equiv 3 \pmod 4$. Let ${K\colonequals \FF_p(t)}$, $K'\colonequals \FF_p(\sqrt{t},\sqrt{1-t})$, and $\Sigma\colonequals \{t^{-1}\}$. We thus have $K_\Sigma=\FF_p(\!(t^{-1})\!)$. As for $K'_{\Sigma'}$, we have:
$$K'_{\Sigma'}\colonequals K'\otimes_{\FF_p(t)} \FF_p(\!(t^{-1})\!) \cong \FF_p(\!(t^{-1})\!)[X,Y]/\left(X^2-t,Y^2-(1-t)\right).$$

Note that $X^2-t$ is irreducible over $\FF_p(\!(t^{-1})\!)$ since, otherwise, this would imply the existence of an element with valuation $-\frac{1}{2}$, which is absurd.
Next, observe that ${1-t=(-1)\,t\,(1-t^{-1})}$. We know that $t$ is a square in $\FF_p(\!(t^{-\frac{1}{2}})\!)$. Furthermore, $1-t^{-1}$ is also a square in $\FF_p(\!(t^{-\frac{1}{2}})\!)$ by Hensel's lemma, since it is a square over $\FF_p$ when reduced modulo $t^{-\frac{1}{2}}$. Consequently, $X^2-(1-t)$ is irreducible over $\FF_p(\!(t^{-\frac{1}{2}})\!)$ if and only if $-1$ is not a square. This is indeed the case, since the condition ${p\equiv 3 \pmod 4}$ implies that $-1$ is not a square in $\FF_p$.
We thus deduce that $\Sigma'=\{t^{-\frac{1}{2}}\}$ and that $K'_{\Sigma'}$ is a field isomorphic to $\FF_{p^2}(\!(t^{-\frac{1}{2}})\!)$. It is a tamely ramified extension of $\FF_p(\!(t^{-1})\!)$ with Galois group \linebreak $\ZZ/2\ZZ\times \ZZ/2\ZZ$.
\medskip

Note that $K$ can potentially only ramify at $t$, $1-t$, and $t^{-1}$. The other places are therefore unramified, hence have a cyclic local Galois group, and thus have degree at most $2$. Indeed, at any other place, $t$ and $1-t$ are units, so their reductions over $\FF_p$ are non-trivial. Consequently, the reductions in $\FF_p[X]$ of the polynomials $X^2-t$ and $X^2-(1-t)$ are separable ($p\not = 2$). Thus, they each admit a root in $\FF_p^\mathrm{sep}$, and therefore over the maximal unramified extension by Hensel's lemma. Hence, the local extensions are unramified.

\begin{itemize}[leftmargin=0.5cm]
    \item At $t$, $X^2-t$ is of course irreducible since $t$ is a uniformizer. On the other hand, $X^2-(1-t)$ admits a root over $\FF_p$ when reduced modulo $t$. Hensel's lemma then ensures that it has a root in $\FF_p(\!(t)\!)$. We then obtain:
    $$K'\otimes_{\FF_p(t)} \FF_p(\!(t)\!) \cong \FF_p(\!(t)\!)[X,Y]/\left(X^2-t,Y^2-(1-t)\right) \cong \FF_p(\!(t^{\frac{1}{2}})\!)[X]/\left(X^2-(\sqrt{1-t})^2\right) \cong \left(\FF_p(\!(t^{\frac{1}{2}})\!)\right)^2.$$
    The two local fields of $K'$ lying over $t$ are therefore cyclic extensions of degree $2$.
    \medskip

    \item At $1-t$, the polynomial $X^2-(1-t)$ is irreducible since $1-t$ is a uniformizer. Also, since $t = 1 - (1-t)$, the polynomial $X^2-t$ admits a root over $\FF_p$ when reduced modulo $1-t$. Hensel's lemma then yields one in $\FF_p(\!(1-t)\!)$. We then have:
    \[\begin{adjustbox}{max size={1\textwidth}{1\textheight}}
    $K'\otimes_{\FF_p(t)} \FF_p(\!(1-t)\!) \cong \FF_p(\!(1-t)\!)[X,Y]/\left(X^2-t,Y^2-(1-t)\right) \cong \FF_p\left(\!\left((1-t)^{\frac{1}{2}} \right)\!\right)[X]/\left(X^2-(\sqrt{t})^2\right) \cong \FF_p\left(\!\left((1-t)^{\frac{1}{2}}\right)\!\right)^2.$
    \end{adjustbox} \]
    Thus, the two local fields of $K'$ lying over $1-t$ are cyclic extensions of degree $2$.
\end{itemize}
\end{pp}

\begin{prop}\label{ContreExempleCorpsDeFonction}
    Let $p$ be a prime such that $p\equiv 3 \pmod 4$. 
    There exist a semisimple $\FF_p(t)$-group $G$ and an $\FF_p(t)$-torus $T$ with $T\left(\FF_p(\!(t^{-1})\!)\right)^0=T\left(\FF_p(\!(t^{-1})\!)\right)^1$, both split over $\FF_p(\sqrt{t},\sqrt{1-t})$, such that:
    $$0\not = \ZZ/2\ZZ \overset{\sim}{\leftarrow} G\left(\FF_p(\!(t^{-1})\!)\right)/\overline{G(\FF_p(t))}\overset{\sim}{\rightarrow}G\left(\FF_p(\!(t^{-1})\!)\right)/G(\FF_p(t))\,G\left(\FF_p(\!(t^{-1})\!)\right)^0 \text{\hspace{2pt} and }$$
    $$0\not = \ZZ/2\ZZ \overset{\sim}{\leftarrow} T\left(\FF_p(\!(t^{-1})\!)\right)/\overline{T(\FF_p(t))}\overset{\sim}{\rightarrow}T\left(\FF_p(\!(t^{-1})\!)\right)/T(\FF_p(t))\,T\left(\FF_p(\!(t^{-1})\!)\right)^0.$$
\end{prop}

\begin{proof}
    The proof is identical to that of Proposition~\ref{ContreExempleCorpsDeNombre}.
\end{proof}

\pagebreak

\section{Counterexamples to the surjectivity of the Kottwitz morphism}\label{SectionCEXSurjectivite}

Here, we consider $K=\RR(t)$ and $\Sigma=\{t\}$, so that $K_\Sigma=\RR(\!(t)\!)$, which is denoted hereafter by $\Kh$. \linebreak In this section, we construct two examples of tori $T$ over $K$ that both satisfy the equality $T(\Kh)=T(K)\,T(\Kh)^0$, but for which the Kottwitz morphism $\kappa_{T_{\Kh}}$ is not surjective \linebreak (see Propositions~\ref{1erCEX} and \ref{2eCEX}).

\begin{pp}[{Context of the counterexamples}]\label{LeContexteDesContreExemples}
Let $L\colonequals \CC(u)$, $\Lh\colonequals \CC(\!(u)\!)$, $\Oc_{\Lh}\colonequals  \CC[[u]]$, and $\Oc_{\Kh}\colonequals  \RR[[t]]$, where $u$ is such that $u^2=t$. Denote by $v_{\Lh}$ the valuation of $\Lh$. Consider ${\Gamma\colonequals \Gal(L/K)\cong\Gal(\Lh/\Kh)\cong\ZZ/2\ZZ \times \ZZ/2\ZZ}$ generated by $\sigma:i\mapsto -i$ and $\tau:u\mapsto -u$. Note that $\Khnr = \CC(\!(t)\!)$. In what follows, it is harmless to assume that $I$ denotes $\Gal(\Lh/\Khnr)$.
\medskip

Let $T$ be a $\Kh$-torus split over $\Lh$. Recall that we have an identification of $\Gamma$-modules ${X\otimes L^\times \overset{\sim}{\rightarrow} T(\Lh)}$ given by $\chi\otimes l \mapsto \chi(l)$. Similarly, $X\otimes \Oc_{\Lh}^\times \overset{\sim}{\rightarrow} T(\Lh)^1$. Note also that the decomposition $\Oc_{\Lh}^\times\cong \CC^\times \times U^1$ (where $U^1$ is the subgroup of $\Oc_{\Lh}^\times$ with constant term $1$) is $\Gamma$-invariant and thus induces the $\Gamma$-invariant decomposition $X\otimes \Oc_{\Lh}^\times \overset{\sim}{\rightarrow} (X\otimes \CC^\times) \oplus (X\otimes U^1)$.
\medskip

Recall from~\cite[Proposition 11.1.1]{kaletha_prasad_2023} that we have the following diagram:
\begin{equation}
\begin{tikzcd}
    {T(\Lh)} & X \\
    {T(\Khnr)} & {X_I}
    \arrow["{\id\otimes v_{\Lh}}", from=1-1, to=1-2]
    \arrow["{N_{\Lh/\Khnr}}"'{pos=0.4}, two heads, from=1-1, 
to=2-1]
    \arrow["{\overline{\cdot}}", two heads, from=1-2, to=2-2]
    \arrow["{\kappa_{T_{\Khnr}}}"', two heads, from=2-1, 
to=2-2]
\end{tikzcd}
\end{equation}
where $N$ is the norm and $\overline{\cdot}\colonequals X\rightarrow X_I$ is the canonical projection. Consequently, ${N(T(\Lh)^1)=T(\Khnr)^0}$. This yields the decomposition $T(\Khnr)^0 = N(X\otimes \CC^\times) \oplus N(X\otimes U^1)$. Furthermore, since $\CC^\times$ is a divisible $\tau$-invariant group, we have $N(X\otimes \CC^\times)=N(X)\otimes \CC^\times = X^I \otimes \CC^\times$. Finally, since $U^1$ has a structure of a $\QQ$-vector space, $\widehat{H}^0(I,X\otimes U^1)=0$ and thus $N(X\otimes U^1)=(X\otimes U^1)^I$. Whence finally:
\begin{equation}
T(\Khnr)^0\cong(X^I \otimes \CC^\times) \oplus (X\otimes U^1)^I
\end{equation}
Note also that this decomposition is $\Galnr$-invariant.
\end{pp}

\subsection{First counterexample}\label{Section1erCEX}

\begin{pp}[{The Galois module of the norm torus}]\label{Contexte1erCEX}
Consider the norm torus $T\colonequals R^{(1)}_{L/K}(\GG_m)$. Its cocharacter module $X$ is the zero-trace sublattice of $\ZZ[\Gamma]$. In other words, if we denote by $(e_1,e_\tau,e_\sigma,e_{\sigma \tau})$ the natural basis of $\ZZ[\Gamma]$, we have:
$$X\colonequals \{ \sum_{\alpha \in \Gamma} a_\alpha e_\alpha \in \ZZ[\Gamma] \mid \sum_{\alpha \in \Gamma} a_\alpha = 0 \}.$$
One checks that a basis of $X$ is given by $(a,b,c)\colonequals (e_1-e_\tau,e_\tau-e_\sigma,e_\sigma - e_{\sigma\tau})$. The actions of $\sigma$ and $\tau$ on the basis are then given by the following matrices:
$$\text{for } \sigma:\begin{pmatrix}
0 & -1 & 1\\
0 & -1 & 0\\
1 & -1 & 0
\end{pmatrix} \text{ and for } 
\tau:\begin{pmatrix}
-1 & 1 & 0\\
0 & 1 & 0\\
0 & 1 & -1
\end{pmatrix}.
$$

Let us then compute $X_I$. By definition, we have:
\begin{align*}
    \overline{a}&=\tau(\overline{a})=\overline{\tau(a)}=-\overline{a}\\
    \overline{b}&=\tau(\overline{b})=\overline{\tau(b)}=\overline{a+b+c}\\
    \overline{c}&=\tau(\overline{c})=\overline{\tau(c)}=-\overline{c}.
\end{align*}
Whence the relations $2\overline{a}=0$ and $\overline{a}=\overline{c}$. Consequently, $X_I\cong \ZZ/2\ZZ \times \ZZ$ with basis $(\overline{a},\overline{b})$. \linebreak Furthermore, the action of $\sigma$ on $X_I$ is then given by the relations:
\begin{align*}
    \sigma(\overline{a})&=\overline{\sigma(a)}=\overline{c}=\overline{a}\\
    \sigma(\overline{b})&=\overline{\sigma(b)}=\overline{-(a+b+c)}=-\overline{b}.
\end{align*}
We then deduce that $(X_I)^{\Galnr}\cong\ZZ/2\ZZ$ with generator $\overline{a}$.
\medskip

This time, let us compute $X^I$. Take $n_1a+n_2b+n_3c \in X^I$. We have:
$$n_1a+n_2b+n_3c=\tau(n_1a+n_2b+n_3c)=-n_1a + n_2(a+b+c)-n_3c=(n_2-n_1)a+n_2b+(n_2-n_3)c.$$
We deduce that $2n_1=2n_3=n_2$. Thus $X^I\cong \ZZ$ and a generator is given by $E\colonequals a+2b+c$. \linebreak Moreover, $\sigma(E)=\sigma(a)+2\sigma(b)+\sigma(c)=c-2(a+b+c)+a=-E$. Whence the action of $\sigma$ on $X^I$.
\end{pp}

\begin{prop}\label{1erCEX}
    The norm $K$-torus $T\colonequals R^{(1)}_{L/K}(\GG_m)$ is such that 
    $$0\cong\Ima(\kappa_{T_{\Kh}}) \subsetneq ((T_*)_I)^{\Galnr}\cong \ZZ / 2\ZZ.$$
    In particular, $T(\Kh)=T(\Kh)^0=T(K)\,T(\Kh)^0$. 
\end{prop}

\begin{proof}
Let us resume the context of $\S$\ref{Contexte1erCEX}. Set $l = a(u)$ and compute $N(l)$. We find:
$$N(l)= l \tau(l) = a(u)\left(\tau(a)(\tau(u))\right)=a(u)a^{-1}(-u)=a(u)a((-u)^{-1})=a(-1).$$
Furthermore, the commutativity of diagram~(1) shows that ${\kappa_{T_{\Khnr}}}(a(-1))=\overline{a} \in (X_I)^{\Galnr}$, since $v_{\Lh}(u)=1$. However, $\sigma(a(-1))=\sigma(a)(\sigma(-1))=c(-1)$. Consequently, $a(-1)\not \in T(\Kh)$. Let us show that $\overline{a}$ is not in the image of $\kappa_{T_{\Kh}}$.

Consider the cocycle $z:\alpha \mapsto a(-1) \alpha(a(-1))^{-1}$ with values in $T(\Khnr)^0$ (since ${\sigma(\overline{a})=\overline{a}}$). \linebreak This cocycle is a coboundary if and only if $\overline{a}\in \Ima(\kappa_{T_{\Kh}})$. Note also that:
$$a(-1) \sigma(a(-1))^{-1}=a(-1)c(-1)^{-1}=a(-1)c(-1)=a(-1)b(-1)^2c(-1)=E(-1) \in X^I \otimes \CC^\times.$$

For the sake of contradiction, suppose that $z$ is a coboundary. It can therefore be written as $\alpha \mapsto y \alpha(y)^{-1}$ with $y\in T(\Khnr)^0$. Let us use the decomposition~(2) to write $y=y_0u$. We then have $y \sigma(y)^{-1}=(y_0 \sigma(y_0)^{-1}) (u \sigma(u)^{-1})$.

By identifying the terms in the decomposition, we obtain $E(-1)=y_0 \sigma(y_0)^{-1}$. Furthermore, since $X^I = \langle E \rangle$, we can write $y_0=E(x)$ with $x\in \CC^\times$. Whence finally:
$$E(-1)=y_0 \sigma(y_0)^{-1}=E(x) \sigma(E)(\sigma(x)^{-1})=E(x)E(\sigma(x))=E(|x|^2).$$
We then deduce that $|x|^2 = -1$. This is a contradiction! Therefore, $z$ is not a coboundary and $\overline{a}\not \in \Ima(\kappa_{T_{\Kh}})$.

Since $\ZZ/2\ZZ \cong (X_I)^{\Galnr}=\langle\overline{a}\rangle$, we deduce that $\Ima(\kappa_{T_{\Kh}})=0$, and thus $T(\Kh)=T(\Kh)^0$. \linebreak The decomposition $T(\Kh)=T(K)\,T(\Kh)^0$ then follows trivially. 
\end{proof}

\subsection{Second counterexample}

\begin{pp}[{The Galois module of an induced torus over $\CC(t)$}]\label{Contexte2eCEX}
Let us now consider the lattice $X=\ZZ^4$ whose standard basis is denoted by $(e,e',f,f')$. We associate with $\sigma$ and $\tau$ the following matrices:
$$\text{for } \sigma:\begin{pmatrix}
A & 0 \\
I_2-A & -A 
\end{pmatrix} \text{ and for } 
\tau:\begin{pmatrix}
A & 0 \\
0 & A 
\end{pmatrix}
\text{, where }
A=
\begin{pmatrix}
0 & 1 \\
1 & 0 
\end{pmatrix} 
.$$
An immediate block computation shows that this defines a $\Gamma$-module structure on $X$. \linebreak We then deduce a unique $K$-torus $T$ such that $T_*=X$. By construction, $T_{\CC(t)}$ is induced.
\medskip

Let us compute $X_I$. Since $\tau(e)=e'$ and $\tau(f)=f'$, we find that $X_I$ is defined by the relations $\overline{e}=\overline{e'}$ and $\overline{f}=\overline{f'}$, so that $X_I \cong \ZZ^2$ with basis $(\overline{e},\overline{f})$. Its $\Galnr$-action is given by:
\begin{align*}
    \sigma(\overline{e})&=\overline{\sigma(e)}=\overline{e'+f-f'}=\overline{e}\\
    \sigma(\overline{f})&=\overline{\sigma(f)}=\overline{-f'}=-\overline{f}.
\end{align*}
We then deduce that $(X_I)^{\Galnr}\cong\ZZ$ with generator $\overline{e}$.
\medskip

Now let us compute $X^I$. Take $ne+n'e'+mf+m'f'\in X^I$. We have:
$$ne+n'e'+mf+m'f'=\tau(ne+n'e'+mf+m'f')=n'e+ne'+m'f+mf'.$$
Whence $n=n'$ and $m=m'$. Thus $ne+n'e'+mf+m'f'=n(e+e')+m(f+f')$. We then deduce that $X^I\cong \ZZ^2$ and admits a basis given by $(e+e',f+f')$. The action of $\Galnr$ on $X^I$ is given by:
\begin{align*}
    \sigma(e+e')&=\sigma(e)+\sigma(e')=(e'+f-f')+(e-f+f')=e+e'\\
    \sigma(f+f')&=\sigma(f)+\sigma(f')=-f'-f=-(f+f').
\end{align*}
\end{pp}

\begin{prop}\label{2eCEX}
    The $K$-torus $T$ induced over $\CC(t)$ introduced in $\S$\ref{Contexte2eCEX} is such that
    $$2\ZZ\cong\Ima(\kappa_{T_{\Kh}}) \subsetneq ((T_*)_I)^{\Galnr}\cong \ZZ.$$
    Moreover, $T(\Kh)=T(K)\,T(\Kh)^0$ and $T(\Khnr)^0=T(\Khnr)^1$ (i.e., $(T_*)_I$ is free).
\end{prop}

\begin{proof}
Let us resume the context of $\S$\ref{Contexte2eCEX}. Set $l = e(iu)$ and compute $N(l)$. We find:
$$N(l)= l \tau(l) = e(iu)\left(\tau(e)(\tau(iu))\right)=e(iu)e'(-iu)=(e+e')(iu)e'(-1).$$
Furthermore, the commutativity of diagram~(1) shows that ${\kappa_{T_{\Khnr}}}(N(l))=\overline{e} \in (X_I)^{\Galnr}$, since $v_{\Lh}(iu)=1$. Let us show that $\overline{e}$ is not in the image of $\kappa_{T_{\Kh}}$.

Consider the cocycle $z:\alpha \mapsto N(l) \alpha(N(l))^{-1}$ with values in $T(\Khnr)^0$ (since ${\sigma(\overline{e})=\overline{e}}$). \linebreak This cocycle is a coboundary if and only if $\overline{e}\in \Ima(\kappa_{T_{\Kh}})$. Note also that:
$$\sigma(N(l))=\sigma((e+e')(iu)e'(-1))=\sigma(e+e')(-iu)\sigma(e')(-1)=(e+e')(-iu)(e-f+f')(-1).$$
Consequently:
\begin{align*}
    N(l)\sigma(N(l))^{-1}&=(e+e')(iu)e'(-1)\left((e+e')(-iu)(e-f+f')(-1)\right)^{-1} \\
    &=e'(-1)(e+e')(-1)(e-f+f')(-1) = (2e+2e'-f+f')(-1) = (f+f')(-1)\in X^I \otimes \CC^\times.
\end{align*}

For the sake of contradiction, suppose that $z$ is a coboundary. It can therefore be written as $\alpha \mapsto y \alpha(y)^{-1}$ with $y\in T(\Khnr)^0$. Let us use the decomposition~(2) to write $y=y_0u$. We then have $y \sigma(y)^{-1}=(y_0 \sigma(y_0)^{-1}) (u \sigma(u)^{-1})$.

By identifying the terms in the decomposition, we obtain $(f+f')(-1)=y_0\sigma(y_0)^{-1}$. Furthermore, since $\ZZ(e+e')$ and $\ZZ(f+f')$ are stable under $\sigma$, and since $(e+e',f+f')$ is a basis of $X^I$, we have $y_0\in \ZZ (f+f') \otimes \CC^\times$. We can thus write $y_0=(f+f')(x)$ for $x\in \CC^\times$. Whence finally:
$$(f+f')(-1)=y_0 \sigma(y_0)^{-1}=(f+f')(x) \sigma(f+f')(\sigma(x)^{-1})=(f+f')(x)(f+f')(\sigma(x))=(f+f')(|x|^2).$$
We then deduce that $|x|^2 = -1$. This is a contradiction! Therefore, $z$ is not a coboundary and $\overline{e}\not \in \Ima(\kappa_{T_{\Kh}})$.

However, $2 \overline{e} \in \Ima(T(K)\rightarrow (X_I)^{\Galnr})$. Indeed, $N(e(t))=(e+e')(t)\in T(K)$ and $v_{\Lh}(t)=2$. It follows that $\Ima(T(K)\rightarrow (X_I)^{\Galnr})=\Ima(\kappa_{T_{\Kh}})$, that is to say $T(\Kh)=T(K)\,T(\Kh)^0$.
\end{proof}

\begin{rema}
    Let $T$ be a $\Kh$-torus and $\mathcal{T}$ its Néron model. We have:
    \begin{align*}
        \coker(\kappa_T)&\overset{\sim}{\to}\Ker\left(H^1(\Galnr,T(\Khnr)^0)\to H^1(\Galnr,T(\Khnr)\right)\\
        {} & \overset{\sim}{\to}\Ker\left(H^1(\Galnr,T(\Khnr)^0)\to H^1(\Gal(\Kh^\mathrm{sep}/\Kh),T(\Kh^\mathrm{sep})\right)\overset{\sim}{\to}\Ker\left(H^1(\Oc_{\Kh},\mathcal{T}^0)\to H^1(\Kh,T)\right).
    \end{align*}
    
    The first isomorphism arises from the exact sequence in cohomology induced by $\kappa_{T_{\Khnr}}$, the second comes from the inflation-restriction exact sequence~\cite[I.\S5.8.a)]{CohoGalois}, and the last one is obtained via~\cite[2.9.2.(2)]{GilleSemisimpleSchemas} and \cite[XXIV, Proposition 8.1.(i)]{SGA3}.

    Consequently, our two previous counterexamples also induce a counterexample to a Grothendieck--Serre type problem for Néron models of tori.
\end{rema}

\section{Birational properties}\label{SectionBirat}

In this section, $K$ is no longer necessarily a global field. We prove the following theorem:

\begin{theo}
    Let $T$ be a $K$-torus, and let $1\rightarrow S \rightarrow P \rightarrow T \rightarrow 1$ be a flasque resolution of $T$. Let $(S_*)_{I,\mathrm{free}}$ denote the quotient of $(S_*)_{I}$ by its torsion part. We have the following natural isomorphisms:
    $$\coker(\kappa_{T_{K_\Sigma}}) \overset{\sim}{\rightarrow} \coker \left(H^1(K_\Sigma,S)\rightarrow H^1(\Gamma^\mathrm{unr},(S_*)_{I,\mathrm{free}}) \right),$$
    $$\coker(T(K)\rightarrow ((T_*)_I)^{\Galnr}) \overset{\sim}{\rightarrow} \coker \left(H^1(K,S)\rightarrow H^1(\Gamma^\mathrm{unr},(S_*)_{I,\mathrm{free}}) \right),\text{\hspace{2pt} and }$$
    $$T(K_\Sigma)/T(K)\, T(K_\Sigma)^0 \overset{\sim}{\rightarrow} \coker \left(H^1(K,S)\rightarrow \Ima(H^1(K_\Sigma,S)\rightarrow H^1(\Gamma^\mathrm{unr},(S_*)_{I,\mathrm{free}}) \right).$$
    In particular, these three quantities are stable birational invariants of $T$.
\end{theo}

\begin{proof}
    Observe that we have the exact sequence:
    \[\begin{tikzcd}
        H_1(I,T_*) & {(S_*)_I} & {(P_*)_I} & {(T_*)_I} & 1.
        \arrow[from=1-1, to=1-2]
        \arrow[from=1-4, to=1-5]
        \arrow[from=1-2, to=1-3]
        \arrow[from=1-3, to=1-4]
    \end{tikzcd}\]
    Since $(P_*)_I$ is free and $H_1(I,T_*)$ is torsion, it immediately follows:
    \[\begin{tikzcd}
        1 & {(S_*)_{I,\mathrm{free}}} & {(P_*)_I} & {(T_*)_I} & 1.
        \arrow[from=1-1, to=1-2]
        \arrow[from=1-4, to=1-5]
        \arrow[from=1-2, to=1-3]
        \arrow[from=1-3, to=1-4]
    \end{tikzcd}\]
    Note that $(P_*)_I$ is a permutation module since $P_*$ is. Furthermore, $\kappa_{P_{K_\Sigma}}$ is surjective, as this holds for $\GG_m$. We then obtain the following diagram with exact rows:
    \[\begin{tikzcd}
         {...} & {P(K_\Sigma) } & {T(K_\Sigma) } & 
    {H^1(K_\Sigma,S)} & {H^1(K_\Sigma,P)=1} \\
         {...} & 
    {((P_*)_I)^{\Gamma^{\mathrm{unr}}}} & 
    {((T_*)_I)^{\Gamma^{\mathrm{unr}}}} & 
    {H^1(\Gamma^{\mathrm{unr}},(S_*)_{I,\mathrm{free}})} & 
    {H^1(\Gamma^{\mathrm{unr}},(P_*)_{I})=1}
        \arrow[from=1-1, to=1-2]
        \arrow[from=1-2, to=1-3]
        \arrow[from=1-2, to=2-2]
        \arrow[from=1-3, to=1-4]
        \arrow["{\kappa_{P_{K_\Sigma}}}", two heads, from=1-2, to=2-2]
        \arrow[from=1-4, to=1-5]
        \arrow["{\kappa_{T_{K_\Sigma}}}", from=1-3, to=2-3]
        \arrow[from=1-4, to=2-4]
        \arrow[from=2-1, to=2-2]
        \arrow[from=2-2, to=2-3]
        \arrow[from=2-3, to=2-4]
        \arrow[from=2-4, to=2-5]
    \end{tikzcd}\]
    A diagram chase then yields the first isomorphism.
    \medskip

    Moreover, $P(K)\rightarrow (P_*)_{I}$ is also surjective because $P(K)$ is dense in $P(K_\Sigma)$. Replacing $K_\Sigma$ by $K$ in the first row of the previous diagram, the same diagram chase then yields the second \linebreak isomorphism.
    \medskip

    The last isomorphism is obtained by observing that
    \begin{align*}
        T(K_\Sigma)/T(K)\,T(K_\Sigma)^0 & \overset{\sim}{\rightarrow}(T(K_\Sigma)/T(K_\Sigma)^0)/(T(K)/T(K)\cap T(K_\Sigma)^0)\\
        &\overset{\sim}{\rightarrow} \Ima(\kappa_{T_{K_\Sigma}})/\Ima(T(K)\rightarrow (T_*)_I)\\
        &\overset{\sim}{\rightarrow}\coker(T(K)\rightarrow (T_*)_I)/\coker(\kappa_{T_{K_\Sigma}})\\
        &\overset{\sim}{\rightarrow}  \coker \left(H^1(K,S)\rightarrow H^1(\Gamma^\mathrm{unr},(S_*)_{I,\mathrm{free}}) \right)/\coker \left(H^1(K_\Sigma,S)\rightarrow H^1(\Gamma^\mathrm{unr},(S_*)_{I,\mathrm{free}}) \right)\\
        &\overset{\sim}{\rightarrow}\coker \left(H^1(K,S)\rightarrow \Ima(H^1(K_\Sigma,S)\rightarrow H^1(\Gamma^\mathrm{unr},(S_*)_{I,\mathrm{free}}) \right).
    \end{align*}
    
    Finally, observe that these three quantities are invariant if we replace $S$ by $S\times P_0$, where $P_0$ is a quasi-trivial torus. We then conclude that these quantities are stable birational invariants of $T$ according to~\cite[\S 4]{VoskresenskiiBook}.
\end{proof}

\section*{Acknowledgments}

The author thanks Philippe Gille and Ralf Köhl for their support, their guidance and their reading of the present article.

The author is also grateful to the Studienstiftung des deutschen Volkes for financially supporting this project. He was also supported by the project ''Group schemes, root systems, and related representations'' founded by the European Union - NextGenerationEU through Romania’s National Recovery and Resilience Plan (PNRR) call no. PNRR-III-C9-2023-I8, Project CF159/31.07.2023, and coordinated by the Ministry of Research, Innovation and Digitalization (MCID) of Romania. 

\printbibliography

\end{document}


\title[Réponse à une question de décomposition sur les tores]{Réponse à une question de décomposition sur les tores posée par Colliot-Thélène et Sansuc}

\author[A.\ Zidani]{Anis Zidani}
\address{Institut de Mathématiques de Jussieu-Paris Rive Gauche, 
Sorbonne Université, Campus Pierre et Marie Curie,
4, place Jussieu, 75252 Paris Cedex 05, France}

\date{\today}

\begin{abstract} 
    L'objectif de cette note est de présenter une stratégie simple pour répondre négativement à une question de décomposition sur les tores posée par Colliot-Thélène et Sansuc dans l'article \textit{Principal Homogeneous Spaces under Flasque Tori: Applications} de 1987. On en déduit alors un tore $T$ sur $\QQ$ et un nombre premier $p$ tel que $T(\ZZ_p)\,T(\QQ)\not=T(\QQ_p)$, où $T(\ZZ_p)$ désigne le sous-groupe compact maximal de $T(\QQ_p)$. 
\end{abstract}

\maketitle

\noindent{\bf Mots clés :} Tores algébriques, Approximation faible, Morphisme de Kottwitz.

\medskip

\noindent{\bf MSC : 11E72, 14G25, 20G30.}

\bigskip

\tableofcontents

\section*{Introduction}

L'objectif de cette note est de répondre à la question suivante :

\begin{ques}[{\cite[Remark 8.3]{CTS}}]\label{QuestionCT}
    Soit $K$ un corps global, $\Kh$ un complété de $K$ pour une valuation discrète, et $\Rh$ l'anneau de valuation discrète associé. Considérons $T$, un tore sur $K$. Désignons par $T(\Rh)$, le sous-groupe compact maximal de $T(\Kh)$. Est-ce que l'égalité
    $$T(\Rh)\, T(K) = T(\Kh)$$
    est satisfaite ?
\end{ques}

\begin{pp}[{Contexte}]
La question \ref{QuestionCT} a été posée pour la première fois en 1987 par Bruhat et Tits à Colliot-Thélène et Sansuc. Dans ce même papier, ils démontrent le cas où $T$ est un $K$-tore non ramifié en utilisant la théorie des résolutions flasques (cf. \cite[Proposition 8.1.(i)]{CTS}). 

Plus tard, en 2007 dans \cite{CTSuresh}, Colliot-Thélène et Suresh ont étudié des variantes du problème : le cas où $R$ est un anneau semi-local (avec donc plusieurs valuations à considérer en même temps), et le cas d'une décomposition « $RT(\Kh)\,T(\Rh) = T(\Kh)$ », où $RT(\Kh)$ est le sous-groupe des éléments $R$-triviaux de $T(\Kh)$ (cf. \cite[7.1.~Définition]{BourbakiGille}). Ils ont réussi à trouver des contre-exemples pour les deux situations. Ils ont également donné une liste de cas où la décomposition « $RT(\Kh)\,T(\Rh) = T(\Kh)$ » est valable (cf. \cite[Proposition 2.1]{CTSuresh}). Cependant, même après ces travaux, la question \ref{QuestionCT} restait ouverte. 
\end{pp}

Dans cette note, on répond par la négative à la question \ref{QuestionCT} (cf. les propositions \ref{ContreExempleCorpsDeNombre} et \ref{ContreExempleCorpsDeFonction}). Plus généralement, on développe une stratégie systématique pour construire des tores ne satisfaisant pas la décomposition demandée. L'ingrédient crucial pour cela est le morphisme de Kottwitz, qui n'avait pas été utilisé jusqu'à présent pour ce type de problème.

\begin{pp}[{Le morphisme de Kottwitz}]\label{paragrapheKottiwtz}
Supposons cette fois que $\Kh$ désigne un corps valué discrètement hensélien de valuation $v$ et dont le corps résiduel est parfait. Le morphisme de Kottwitz peut-être vu comme une généralisation fonctorielle du morphisme de valuation
$$\begin{aligned}
    \kappa_{\GG_m} \colon \Kh^\times & \longrightarrow \ZZ \\
     k & \longmapsto v(k)
\end{aligned}$$

à tout groupe réductif $G$ sur $\Kh$. Il a été défini par Kottwitz dans \cite[Section 7]{Kottwitz}.
\medskip

Par exemple, le morphisme de Kottwitz pour $\GL_n$ sur $\Kh$ est :
$$\begin{aligned}
    \kappa_{\GL_n} \colon \GL_n(\Kh) & \longrightarrow \ZZ \\
     M & \longmapsto v(\mathrm{det}(M)),
\end{aligned}$$

tandis que pour $\mathrm{PGL}_n$ sur $\Kh$, on a :
$$\begin{aligned}
    \kappa_{\mathrm{PGL}_n}\colon \mathrm{PGL}_n(\Kh) & \longrightarrow \ZZ / n \ZZ  \\
     \overline{M} & \longmapsto v(\mathrm{det}(\overline{M})) \pmod n .
\end{aligned}$$

Le morphisme est défini grâce au groupe fondamental algébrique (cf. \cite[Definition 11.3.2]{kaletha_prasad_2023}), que l'on note $\pi_1(G)$. Il s'agit d'un $\Gal(\Kh^\mathrm{sep}/\Kh)$-module (où $\Kh^\mathrm{sep}$ est la clôture séparable de $\Kh$). \linebreak Pour un $\Kh$-tore $T$, on a $\pi_1(T)=T_*$, le module galoisien des cocaractères de $T$.
\medskip

Considérons $\Khnr$, l'extension maximale non ramifiée de $\Kh$. Notons $\Galnr\colonequals \Gal(\Khnr/\Kh)$ et ${I\colonequals \Gal(\Kh^\mathrm{sep}/\Khnr)}$. Considérons un groupe réductif $G$ sur $\Kh$. Il s'avère que le morphisme de Kottwitz de $G$ sur $\Khnr$ est un morphisme de $\Galnr$-modules : $\kappa_{G_{\Khnr}}:G(\Khnr)\rightarrow \pi_1(G)_I$.
\medskip

Le morphisme de Kottwitz de $G$ sur $\Kh$ est alors le morphisme $\kappa_G:G(\Kh)\rightarrow (\pi_1(G)_I)^{\Galnr}$ obtenu à partir de $\kappa_{G_{\Khnr}}$ en prenant les $\Galnr$-invariants. Il jouit de nombreuses propriétés :

\begin{itemize}
    \item Il est surjectif lorsque $\Kh$ est un corps local (cf. \cite[Corollary 11.7.6]{kaletha_prasad_2023}).
    \item Il est trivial pour tout groupe semi-simple simplement connexe (car est de $\pi_1$ trivial).
    \item Son noyau est $G(\Kh)^0$, le sous-groupe ouvert engendré par les sous-groupes parahoriques de $G$ (cf. \cite[Proposition 11.5.4]{kaletha_prasad_2023}).
    \item L'image réciproque de la partie de torsion de $(\pi_1(G)_I)^{\Galnr}$ est $G(\Kh)^1$, l'intersection des noyaux de $g\mapsto v(\chi(g))$ pour $\chi$ parcourant les caractères de $G$ (cf. \cite[Lemma 11.5.2]{kaletha_prasad_2023}).
\end{itemize}

Par ailleurs, pour un $\Kh$-tore $T$, il s'avère que $T(\Kh)^0$ est le sous-groupe des points du modèle de Néron connexe de $T$ (cf. \cite[Corollary B.8.7]{kaletha_prasad_2023}), et $T(\Kh)^1$ est le sous-groupe borné maximal de $T(\Kh)$.
\medskip
\end{pp}

\begin{pp}[{Stratégie et plan de l'article}]
Notre stratégie pour répondre à la question \ref{QuestionCT} est en deux étapes.
En section \ref{SectionLemmesKottwitz} on prouve que, pour un groupe réductif $G$ sur $\Kh$ (qui est cette fois un corps local) déployé par une extension dont le sous-groupe d'inertie sauvage est cyclique, l'application ${G(\Kh)/\overline{G(K)}\rightarrow G(\Kh)/G(K)\,G(\Kh)^0}$ est un isomorphisme (cf. la proposition \ref{ThmCorpsNbr}). Ceci est prouvé grâce au morphisme de Kottwitz. Ensuite, en section \ref{SectionCEX} on trouve un tore $T$ sur $K$ vérifiant l'hypothèse précédente (en fait déployée par une extension modérément ramifiée) tel que ${T(\Kh)^1=T(\Kh)^0}$ et ${\overline{T(K)}\not = T(\Kh)}$ (c'est-à-dire que $T$ ne satisfait pas l'approximation faible). L'approximation faible sur les corps de nombres est très bien comprise grâce aux travaux de Sansuc (cf. \cite{Sansuc}). Borovoï fournit dans \cite{borovoi2025defectweakapproximationreductive} une version incluant les corps de fonctions.
On en déduit alors dans les propositions \ref{ContreExempleCorpsDeNombre} et \ref{ContreExempleCorpsDeFonction} une réponse négative à la question \ref{QuestionCT} (et même une réponse négative à une question analogue dans le cas d'un groupe semi-simple). Précisons que cette stratégie fonctionne également pour étudier la variante semi-locale du problème.

En complément, on discute dans les sections \ref{SectionCEXSurjectivite} et \ref{SectionBirat} de cet article quelques problèmes supplémentaires autour du morphisme de Kottwitz.
D'une part, à la section \ref{SectionCEXSurjectivite}, on donne deux exemples de tores sur $\RR(\!(t)\!)$ dont les morphismes de Kottwitz ne sont pas surjectifs. La surjectivité n'est donc réalisée que pour des corps spécifiques. D'autre part, on montre en section \ref{SectionBirat} que, pour un $K$-tore $T$, les groupes ${\coker(\kappa_{T_{\Kh}})}$, ${\coker(T(K)\rightarrow ((T_*)_I)^{\Galnr})}$ et ${T(\Kh)/T(K)\, T(\Kh)^0}$ sont des invariants birationnels stables de $T$.
\end{pp}

\section*{Notations}

Soit $K$ un corps. Notons :
\begin{itemize}
    \item $\Omega_K$ l'ensemble des places de $K$, $\Sigma$ un ensemble fini de $\Omega_K$ et $\overline{\Sigma}\colonequals \Omega_K \backslash \Sigma$,
    \item $K_v$ le complété de $K$ en une place $v\in \Sigma$ et $K_\Sigma\colonequals  \prod_{v\in \Sigma}K_v$ (on y inclut $K$ diagonalement),
    \item $K_v^{\mathrm{nr}}$, l'extension maximale non ramifiée de $K_v$ et $K_\Sigma^\mathrm{nr}\colonequals \prod_v K_v^\mathrm{nr}$,
    \item $\Galnr_v\colonequals  \Gal(K_v^\mathrm{nr}/K_v)$, $I_v$ le sous-groupe d'inertie de $K_v$, $\Galnr\colonequals \prod_v \Galnr_v$ et $I\colonequals \prod_v I_v$.
\end{itemize}
On suppose que, pour tout $v\in \Sigma$, le corps résiduel de $K_v$ est \emph{parfait}. Les notations, propriétés, définitions que l'on introduit dans cette note pour un corps complet se généralisent au cas de $K_\Sigma$ en raisonnant facteur par facteur. Cela est implicite tout au long du document.

On reprend également les notations $(-)^\mathrm{sep}$, $\pi_1(-)$, $(-)_*$, $\kappa_{(-)}$, $(-)^0$ et $(-)^1$ du $\S$\ref{paragrapheKottiwtz}.
\medskip

Dans les parties \ref{SectionLemmesKottwitz} et \ref{SectionCEX}, le corps $K$ est supposé être \emph{global}.

\section{Quelques lemmes sur le morphisme de Kottwitz}\label{SectionLemmesKottwitz}

L'objectif de cette partie est d'établir quelques résultats cruciaux pour l'élaboration de notre réponse à la question \ref{QuestionCT}, lesquels font intervenir le morphisme de Kottwitz. On montre en particulier que, sous certaines conditions, le problème se ramène à une question d'approximation faible.

\begin{pp}[{Inertie sauvage cyclique}] 
On dit qu'un groupe réductif sur $K_\Sigma$ est à \emph{inertie sauvage cyclique} si pour tout $v\in \Sigma$, il existe une extension galoisienne déployante de $K_v$ telle que son sous-groupe d'inertie sauvage soit cyclique (c'est en particulier le cas si l'extension est modérément ramifiée). On observe que c'est équivalent à demander que le sous-groupe d'inertie soit un groupe métacyclique (c'est à dire tel que ses sous-groupes de Sylow soient cycliques).
\end{pp}

\begin{lemm}\label{FactDirectTore}
    Soit $T$ un tore sur $K_\Sigma$ qui est un facteur direct d'un tore quasi-trivial sur $K_\Sigma^\mathrm{nr}$. \linebreak Alors $\pi_1(T)_I$ est libre et $T(K_\Sigma)^0=T(K_\Sigma)^1$.
\end{lemm}

\begin{proof}
    La construction du foncteur $\pi_1(-)_I$ se comporte bien avec les produits directs, de sorte que $\pi_1(T)_I$ est un facteur direct du «$\pi_1(-)_I$» d'un tore quasi-trivial, qui est libre. Par conséquent, $\pi_1(T)_I$ est libre. D'après \cite[Corollary 11.1.6]{kaletha_prasad_2023}, on a alors $T(K_\Sigma)^0=T(K_\Sigma)^1$.
\end{proof}

\begin{lemm}\label{FlasqueTame}
    Soit $T$ un $K_\Sigma$-tore à inertie sauvage cyclique qui est soit flasque, soit coflasque. Alors $T$ est un facteur direct d'un tore quasi-trivial sur $K_\Sigma^\mathrm{nr}$, le groupe $\pi_1(T)_I$ est libre et $T(K_\Sigma)^0=T(K_\Sigma)^1$.
\end{lemm}

\begin{proof}
    Le théorème d'Endo et Miyata (cf. \cite[0.5]{CTS}) nous montre qu'un tore flasque ou coflasque est un facteur direct d'un tore quasi-trivial lorsque ce tore est déployé par une extension métacyclique. Donc $T$ est un facteur direct d'un tore quasi-trivial sur $K_\Sigma^\mathrm{nr}$. Le reste découle alors du lemme \ref{FactDirectTore}.
\end{proof}

\pagebreak

\begin{lemm}\label{SuiteKottwitz}
    Soit $1\rightarrow A \rightarrow B \rightarrow C \rightarrow 1$ une suite exacte de $K_\Sigma$-groupes réductifs. Supposons que $A$ soit un tore et que le groupe $\pi_1(A)_I$ soit libre. On a alors les suites exactes suivantes :
    \[\begin{tikzcd}
    	1 & {\pi_1(A)_I} & {\pi_1(B)_I} & {\pi_1(C)_I} & 1
    	\arrow[from=1-1, to=1-2]
    	\arrow[from=1-4, to=1-5]
    	\arrow[from=1-2, to=1-3]
    	\arrow[from=1-3, to=1-4]
    \end{tikzcd}\]
    et
    \[\begin{tikzcd}
    	1 & {A(K_\Sigma)^0} & {B(K_\Sigma)^0} & {C(K_\Sigma)^0} & 1.
    	\arrow[from=1-1, to=1-2]
    	\arrow[from=1-4, to=1-5]
    	\arrow[from=1-2, to=1-3]
    	\arrow[from=1-3, to=1-4]
    \end{tikzcd}\]
\end{lemm}

\begin{proof}
    Rappelons que le foncteur $\pi_1(-)$ est exact (cf. \cite[Theorem 3.8]{borovoi2013algebraicfundamentalgroupreductive}). En prenant les coinvariants par $I$, on obtient :
    \[\begin{tikzcd}
    	H_1(I,\pi_1(A)) & {\pi_1(A)_I} & {\pi_1(B)_I} & {\pi_1(C)_I} & 1.
    	\arrow[from=1-1, to=1-2]
    	\arrow[from=1-4, to=1-5]
    	\arrow[from=1-2, to=1-3]
    	\arrow[from=1-3, to=1-4]
    \end{tikzcd}\]
    Puisque $\pi_1(A)_I$ est libre d'après le lemme \ref{FactDirectTore}, l'image de $H_1(I,\pi_1(A))$ (qui est un groupe de torsion) dans $\pi_1(A)_I$ est triviale. D'où la première suite exacte.
    
    Quant à la seconde suite exacte, puisque le corps résiduel est parfait, le théorème de Steinberg (cf. \cite[Theorem 2.3.3.(1)]{kaletha_prasad_2023}) nous donne $H^1(I,A(K_\Sigma^\mathrm{sep}))=1$. Nous avons alors :
    \[\begin{tikzcd}
    	1 & {A(K_\Sigma^\mathrm{nr})} & {B(K_\Sigma^\mathrm{nr})} & {C(K_\Sigma^\mathrm{nr})} & H^1(I,A(K_\Sigma^\mathrm{sep})) = 1.
    	\arrow[from=1-1, to=1-2]
    	\arrow[from=1-4, to=1-5]
    	\arrow[from=1-2, to=1-3]
    	\arrow[from=1-3, to=1-4]
    \end{tikzcd}\]
    En combinant les deux suites exactes précédentes grâce au morphisme de Kottwitz, et en utilisant le lemme du Serpent, on obtient :
    \[\begin{tikzcd}
    	1 & {A(K_\Sigma^\mathrm{nr})^0} & {B(K_\Sigma^\mathrm{nr})^0} & {C(K_\Sigma^\mathrm{nr})^0} &  1.
    	\arrow[from=1-1, to=1-2]
    	\arrow[from=1-4, to=1-5]
    	\arrow[from=1-2, to=1-3]
    	\arrow[from=1-3, to=1-4]
    \end{tikzcd}\]
    En prenant les invariants par $\Galnr$ et en utilisant le fait que $H^1(\Galnr,A(K_\Sigma^\mathrm{nr})^0)=1$ d'après \cite[Lemma 11.7.1]{kaletha_prasad_2023}, on obtient alors :
    \[\begin{tikzcd}
    	1 & {A(K_\Sigma)^0} & {B(K_\Sigma)^0} & {C(K_\Sigma)^0} &  H^1(\Galnr,A(K_\Sigma^\mathrm{nr})^0)=1.
    	\arrow[from=1-1, to=1-2]
    	\arrow[from=1-4, to=1-5]
    	\arrow[from=1-2, to=1-3]
    	\arrow[from=1-3, to=1-4]
    \end{tikzcd}\]
    D'où la suite exacte désirée.
\end{proof}

\begin{prop}\label{ThmCorpsNbr}
    Soit $G$ un groupe réductif sur $K$. Si le groupe $G_{K_\Sigma}$ est à inertie sauvage cyclique, alors l'application naturelle suivante est un isomorphisme :
    $$G(K_\Sigma)/\overline{G(K)}\overset{\sim}{\rightarrow}G(K_\Sigma)/G(K)\,G(K_\Sigma)^0.$$
\end{prop}

\begin{proof}
    Comme $G(K_\Sigma)^0$ est un sous-groupe ouvert de $G(K_\Sigma)$, on a l'inclusion ${\overline{G(K)}\subset G(K)G(K_\Sigma)^0}$. Il suffit alors de prouver que $G(K_\Sigma)^0 \subset \overline{G(K)}$. Remarquons que, puisque le groupe $G_{K_\Sigma}$ est à inertie sauvage cyclique, on doit avoir une résolution flasque (cf. \cite[Proposition-Définition 3.1]{FlasqueGrpLin}) sur $K_\Sigma^\mathrm{nr}$ donnée par ${1 \rightarrow F' \rightarrow H' \rightarrow G_{K_\Sigma^\mathrm{nr}}\rightarrow 1}$ où $F'$ est également à inertie sauvage cyclique. Par conséquent, il s'agit d'un facteur direct d'un tore quasi-trivial d'après le lemme \ref{FlasqueTame}.
    
    Considérons maintenant une résolution flasque de $G$ sur $K$ donnée par $1\rightarrow F \rightarrow H \rightarrow G \rightarrow 1$. Puisque $F$ est unique à un tore quasi-trivial près (cf. \cite[Proposition 3.2.(iii)]{FlasqueGrpLin}), la discussion précédente montre que $F$ est en fait un facteur direct d'un tore quasi-trivial sur $K_\Sigma^\mathrm{nr}$. De plus, \cite[Proposition 2.6]{borovoi2025defectweakapproximationreductive} nous indique que $H$ vérifie l'approximation faible, autrement dit : $\overline{H(K)}=H(K_\Sigma)$. Ceci implique que le sous-groupe $H(K_\Sigma)^0\subset H(K_\Sigma)$ est envoyé sur $\overline{G(K)}$. Mais le lemme \ref{SuiteKottwitz} nous indique que $H(K_\Sigma)^0\rightarrow G(K_\Sigma)^0$ est surjectif, de sorte que $G(K_\Sigma)^0\subset \overline{G(K)}$ comme souhaité.
\end{proof}

\begin{rema}
    Le résultat n'est plus vrai si le groupe n'est plus à inertie sauvage cyclique. Considérons l'extension $L/K\colonequals \QQ(\zeta_8)/\QQ=\QQ(\sqrt{-1},\sqrt{2})/\QQ$ (où $\zeta_8=e^{\frac{2i\pi}{8}}$) et considérons le tore normique $T\colonequals R^{(1)}_{L/K}(\GG_m)$. D'après \cite[11.6.~Example 3.2)]{VoskresenskiiBook}, $\QQ_2(\zeta_8)/\QQ_2$ est une extension biquadratique totalement ramifiée telle que $T(\QQ_2)/\overline{T(\QQ)} = \ZZ/2\ZZ \not = 0$. Or, $\coker(T(\QQ)\rightarrow \Ima(\kappa_{T_{\QQ_2}}))=0$ \linebreak (i.e., $T(\QQ_2)/T(\QQ)\, T(\QQ_2)^0 = 0$). En effet, en réadaptant la section \ref{Section1erCEX} à notre contexte et en reprenant ses notations, on observe que $N(a(1-\zeta_8)) \in T(\QQ)$ et $N(b(1-\zeta_8)) \in T(\QQ)$ s'envoient respectivement sur $\overline{a}$ et $\overline{b}$ par $\kappa_{T_{\QQ_2}}$, et engendrent alors $\Ima(\kappa_{T_{\QQ_2}})$. 
\end{rema}

\section{Approximation faible et création d'exemples}\label{SectionCEX}

Dans cette partie, on calcule des défauts d'approximation faible. Combinés à la proposition  \ref{ThmCorpsNbr}, on répond négativement à la question \ref{QuestionCT} (et à sa variante semi-simple), lorsque $K$ est un corps de nombres (cf. la proposition \ref{ContreExempleCorpsDeNombre}) et lorsque $K$ est un corps de fonctions (cf. la proposition \ref{ContreExempleCorpsDeFonction}).

\begin{lemm}\label{CalculCokerMu}
    Soit $\mu\colonequals R_{K'/K}(\mu_n)/\mu_n$ avec $n\in \NN^\times$ tel que $\mathrm{car}(K) \nmid n$. On a :
    $$\coker \left(H^1(K,\mu)\rightarrow H^1(K_\Sigma,\mu)\right)\overset{\sim}{\rightarrow} \coker\left(\Br(K'/K)[n]\rightarrow \Br(K'_{\Sigma'}/K_\Sigma)[n]\right) \overset{\sim}{\rightarrow} \ZZ/\frac{\pgcd(d_\Sigma,n)}{\pgcd(d_\Sigma,d_{\overline{\Sigma}},n)} \ZZ,$$
    où pour un ensemble de places $S$ de $K$, on note $d_S\colonequals \ppcm_{v\in S}[K'_{v'}:K_v]$, et où pour tout $v\in \Omega_K$, \linebreak on désigne par $v'$ une place de $K'$ au dessus de $v$. 
\end{lemm}

\begin{proof}
    Posons $T\colonequals  R_{K'/K}(\GG_m)/\GG_m$. Observons que $\mu=\ker(T\overset{\times n}{\rightarrow} T)$. En utilisant la suite induite en cohomologie sur $K$ et sur $K_\Sigma$, on a le diagramme suivant à lignes exactes :
    \[\begin{tikzcd}
    	1 & {T(K)/nT(K)} & {H^1(K,\mu)} & {\Br(K'/K)[n]} & 1 \\
    	1 & {T(K_\Sigma)/nT(K_\Sigma)} & {H^1(K_\Sigma,\mu)} & 
    {\Br(K'_{\Sigma'}/K_\Sigma)[n]} & 1.
    	\arrow[from=1-1, to=1-2]
    	\arrow[from=1-2, to=1-3]
    	\arrow[from=1-2, to=2-2]
    	\arrow[from=1-3, to=1-4]
    	\arrow[from=1-3, to=2-3]
    	\arrow[from=1-4, to=1-5]
    	\arrow[from=1-4, to=2-4]
    	\arrow[from=2-1, to=2-2]
    	\arrow[from=2-2, to=2-3]
    	\arrow[from=2-3, to=2-4]
    	\arrow[from=2-4, to=2-5]
    \end{tikzcd}\]

    Comme $T$ admet par définition une résolution flasque ${1\rightarrow \GG_m \rightarrow R_{K'/K}(\GG_m)\rightarrow T \rightarrow 1}$, il est rétracte $K$-rationnel (cf. \cite[Proposition 7.4]{CTS}). En conséquence, $T(K)$ est dense dans $T(K_\Sigma)$ (cf.\cite[Proposition 8.1.(iii)]{CTS}). Par ailleurs, puisque $\mathrm{car}(K)\nmid n$, la multiplication par $n$ dans $T$ est lisse, de sorte que $nT(K_\Sigma)$ est un ouvert de $T(K_\Sigma)$ (cf. \cite[3.1.2 Lemme]{GGMB}). On a donc ${T(K) \cdot nT(K_\Sigma) = T(K_\Sigma)}$. Ceci donne que $T(K)/nT(K)\rightarrow T(K_\Sigma)/nT(K_\Sigma)$ est surjective.
    \medskip

    Le lemme du serpent appliqué au diagramme précédent donne alors l'isomorphisme :
    $$\coker \left(H^1(K,\mu)\rightarrow H^1(K_\Sigma,\mu)\right)\overset{\sim}{\rightarrow}\coker\left(\Br(K'/K)[n]\rightarrow \Br(K'_{\Sigma'}/K_\Sigma)[n]\right).$$
    Calculons maintenant le membre de droite. Le théorème de Brauer-Hasse-Noether (cf. \cite[Théorème 14.11]{HarariCorpsDeClasses}) restreint à $\Br(K'/K)[n]$ donne la suite exacte suivante :
    \[\begin{tikzcd}
    	1 & {\Br(K'/K)[n]} & {\bigoplus_{v\in \Omega_K} \Br(K'\otimes_K K_v /K_v)[n]} 
    & \frac{1}{n}\ZZ/\ZZ. 
    	\arrow[from=1-1, to=1-2]
    	\arrow[from=1-2, to=1-3]
    	\arrow["{\sum \mathrm{inv}_v}", from=1-3, to=1-4]
    \end{tikzcd}\]
    Considérons les applications partielles $\phi_\Sigma\colonequals \sum_{v\in \Sigma} \mathrm{inv}_v$ et $\phi_{\overline{\Sigma}}\colonequals \sum_{v\in \overline{\Sigma}} \mathrm{inv}_v$. On a donc ${\phi_\Sigma(x)=-\phi_{\overline{\Sigma}}(x)=\phi_{\overline{\Sigma}}(-x)}$ pour $x\in \Br(K'/K)[n]$. En fait, $\Ima((\phi_\Sigma)_{\mid\Br(K'/K)[n]})=\Ima(\phi_\Sigma)\cap \Ima(\phi_{\overline{\Sigma}})$. Ceci permet d'en déduire finalement la suite exacte suivante :
    \[\begin{tikzcd}
    	1 & {\Br(K'/K)[n]} & {\Br(K'_{\Sigma'}/K_\Sigma)[n]} & 
    {\Ima(\phi_\Sigma)/(\Ima(\phi_\Sigma)\cap \Ima(\phi_{\overline{\Sigma}}))}
     & 1.
    	\arrow[from=1-1, to=1-2]
    	\arrow[from=1-2, to=1-3]
    	\arrow["{\phi_\Sigma}", from=1-3, to=1-4]
    	\arrow[from=1-4, to=1-5]
    \end{tikzcd}\]
    Pour terminer, on calcule le cardinal du quotient (qui est cyclique).  Prenons $v$ une place de $K$. Observons que $\Br(K'_{v'}/K_v)=\Br(K'\otimes_K K_v /K_v)$. Aussi, $d_v\colonequals [K'_{v'}:K_v]=| \Br(K'_{v'}/K_v) |$ et l'image de $\Br(K'_{v'} /K_v)$ dans $\QQ/\ZZ$ est $(\frac{1}{d_v}\ZZ)/\ZZ$ (cf. \cite[Corollaire 8.10]{HarariCorpsDeClasses}). Donc l'image de $\Br(K'\otimes_K K_v /K_v)[n]$ est $(\frac{1}{\pgcd(d_v,n)}\ZZ)/\ZZ$. En conséquence, $\Ima(\phi_\Sigma)=\frac{1}{\pgcd(d_\Sigma,n)}\ZZ/\ZZ$ et $\Ima(\phi_{\overline{\Sigma}})=\frac{1}{\pgcd(d_{\overline{\Sigma}},n)}\ZZ/\ZZ$. On observe alors que $\Ima(\phi_{\overline{\Sigma}})\cap \Ima(\phi_\Sigma) = \frac{1}{\pgcd(|\Ima(\phi_\Sigma)|, |\Ima(\phi_{\overline{\Sigma}})|)}\ZZ / \ZZ=\frac{1}{\pgcd(d_\Sigma,d_{\overline{\Sigma}},n)}\ZZ/\ZZ$.
    On conclut finalement que le conoyau est de cardinal $\frac{\pgcd(d_\Sigma,n)}{\pgcd(d_\sigma,d_{\overline{\Sigma}},n)}$ comme voulu.
\end{proof}

\begin{prop}\label{stratégie}
    Soit $K'/K$ une extension finie galoisienne et un $K$-groupe fini étale de type multiplicatif $\mu$ déployé sur $K'$. Posons $\Sigma'\colonequals \left\{w\in \Omega_{K'} \mid \exists v\in \Sigma, w\mid v\right\}$. On a alors $K'_{\Sigma'}=K'\otimes_K K_\Sigma$. Supposons que le sous-groupe d'inertie sauvage de $K'_{\Sigma'}/K_\Sigma$ est cyclique.
    Alors, il existe un $K$-groupe semi-simple $G$ et un $K$-tore $T$ tel que $T(K_\Sigma)^0=T(K_\Sigma)^1$, tous deux déployés sur $K'$, tels que :
    $$G(K_\Sigma)/G(K)\,G(K_\Sigma)^0 \overset{\sim}{\leftarrow} G(K_\Sigma)/\overline{G(K)}\overset{\sim}{\rightarrow}\coker \left(H^1(K,\mu)\rightarrow H^1(K_\Sigma,\mu)\right) \text{\hspace{2pt} et }$$
    $$T(K_\Sigma)/T(K)\,T(K_\Sigma)^0 \overset{\sim}{\leftarrow}T(K_\Sigma)/\overline{T(K)}\overset{\sim}{\rightarrow}\coker \left(H^1(K,\mu)\rightarrow H^1(K_\Sigma,\mu)\right).$$
\end{prop}

\begin{proof}
    Les isomorphismes de gauche sont une conséquence de la proposition \ref{ThmCorpsNbr}. \linebreak En effet, $G$ et $T$ sont à inertie sauvage cyclique puisque déployés sur $K'$, donc sur $K'_{\Sigma'}$.

    Il existe une famille finie d'entiers $(n_i)$ telle que $\mu_{K'}\cong \prod \mu_{n_i}$. On plonge alors $\mu$ dans le centre de $G'\colonequals \mathrm{R}_{K'/K}(\prod \mathrm{SL}_{n_i})$. En posant $G\colonequals  G'/\mu$, on obtient une isogénie centrale de noyau $\mu$, où $G$ est un groupe semi-simple (ceci était déjà connu de Serre dans \cite[p.158]{CohoGalois}). D'autre part, selon \cite[Proposition 1.3]{CTS}, on peut trouver une suite exacte ${1\rightarrow \mu \rightarrow P \rightarrow T \rightarrow 1}$ tel que $P$ est quasi-trivial et $T$ est coflasque. On peut également observer que, par construction, $P$ et $T$ peuvent être choisis comme étant déployés sur $K'$. \cite[Theorem 2.12]{borovoi2025defectweakapproximationreductive} donne alors les isomorphismes de droite.

    Enfin, le lemme \ref{FlasqueTame} montre que $T(K_\Sigma)^0=T(K_\Sigma)^1$.
\end{proof}

\subsection{Exemple dans le cas des corps de nombres}

\begin{pp}[{Étude d'un corps de nombres}]\label{LeCorpsDeNombres}
Posons $K\colonequals \QQ$, $K'\colonequals \QQ(\sqrt{3},\sqrt{11})$ et $\Sigma\colonequals \{3\}$. On a donc $K_\Sigma=\QQ_3$. Pour ce qui est de $K'_{\Sigma'}$, on a :
$$K'_{\Sigma'}\colonequals K'\otimes_\QQ \QQ_3 \cong \QQ[X,Y]/(X^2-3,Y^2-11) \otimes_Q \QQ_3 \cong \QQ_3[X,Y]/(X^2-3,Y^2-11).$$ 
De plus, $X^2-3$ est irréductible dans $\QQ_3$ puisque $3$ est une uniformisante. Par ailleurs, $X^2-11$ n'a pas de racine dans $\ZZ/3\ZZ$, et donc pas dans $\ZZ_3(\sqrt{3})$. Puisque $X^2-11\in \ZZ_3(\sqrt{3})[X]$, il n'a pas non plus de racine dans $\QQ_3(\sqrt{3})$. Par conséquent, $\QQ_3[X,Y]/(X^2-3,Y^2-11)=\QQ_3(\sqrt{3},\sqrt{11})$ est un corps. En conséquence, $\Sigma'=\{\sqrt{3}\}$ et $K'_{\Sigma'}\cong \QQ_3(\sqrt{3},\sqrt{11})$, qui est alors de dimension $4$ et, bien sûr, de groupe de Galois $\ZZ/2\ZZ\times \ZZ/2\ZZ$. Donc $K'_{\Sigma'}/K_\Sigma$ est modérément ramifiée.
\medskip
    
Cherchons les nombres premiers pour lesquels $\QQ(\sqrt{3},\sqrt{11})/\QQ$ est ramifiée. Bien entendu, tous les groupes de décomposition sont cycliques dans le cas non ramifié. D'après \cite[Chapter 2, Exercise 42.(f)]{Marcus}, le discriminant de cette extension est $16\times 3^2 \times 11^2$. Les nombres premiers qui divisent le discriminant, c'est-à-dire les nombres premiers ramifiés, sont alors $2$, $3$ et $11$.

\begin{itemize}[leftmargin=0.5cm]
    \item En $p=2$, on remarque que $P\colonequals X^2-33$ admet $1$ comme racine modulo $8$ (et ${P'(1)=2 \pmod {4}}$). Cela signifie que $|P(1)|_2 \leq \frac{1}{8} < |P'(1)|_2^2 =\frac{1}{2}^2$. La méthode de Newton nous donne alors une racine dans $\QQ_2$. On remarque également que $X^2-3$ n'a pas de racine dans $\ZZ/4\ZZ$, donc pas de racine dans $\ZZ_2$, et par conséquent pas de racine dans $\QQ_2$ puisque $X^2-3\in \ZZ_2[X]$.
    Ainsi, $X^2-3$ est irréductible dans $\QQ_2$. On obtient alors :
    $$K'\otimes_\QQ \QQ_2 \cong \QQ_2[X,Y]/(X^2-3,Y^2-11) \cong \left(\QQ_2[\sqrt{3}]\right)[Y]/(Y^2-(\frac{\sqrt{33}}{\sqrt{3}})^2) \cong \left(\QQ_2(\sqrt{3})\right)^2.$$
    Les deux corps locaux de $K'$ au-dessus de $2$ sont donc des extensions cycliques de degré $2$. 
    \medskip

    \item En $p=11$, $P\colonequals X^2-3$ admet $5$ comme racine modulo $11$ (et $P'(5)\not=0 \pmod {11}$). D'après le lemme de Hensel, $P\colonequals X^2-3$ possède une racine dans $\QQ_{11}$. De plus, $X^2-11$ est irréductible dans $\QQ_{11}$ puisque $11$ est une uniformisante. On a alors :
    $$K'\otimes_\QQ \QQ_{11} \cong \QQ_{11}[X,Y]/(X^2-3,Y^2-11) \cong \left(\QQ_{11}(\sqrt{11})\right)[X]/(X^2-(\sqrt{3})^2) \cong \left(\QQ_{11}(\sqrt{11})\right)^2.$$
    Ainsi, les deux corps locaux de $K'$ au-dessus de $11$ sont des extensions cycliques de degré $2$.
\end{itemize}
\end{pp}

\begin{prop}\label{ContreExempleCorpsDeNombre}
    Il existe un $\QQ$-groupe semi-simple $G$ et un $\QQ$-tore $T$ tel que $T(\QQ_3)^0=T(\QQ_3)^1$, tous deux déployés sur $\QQ(\sqrt{3},\sqrt{11})$, tels que :
    $$0\not = \ZZ/2\ZZ \overset{\sim}{\leftarrow} G(\QQ_3)/\overline{G(\QQ)}\overset{\sim}{\rightarrow}G(\QQ_3)/G(\QQ)\,G(\QQ_3)^0 \text{\hspace{2pt} et }$$
    $$0\not = \ZZ/2\ZZ \overset{\sim}{\leftarrow} T(\QQ_3)/\overline{T(\QQ)}\overset{\sim}{\rightarrow}T(\QQ_3)/T(\QQ)\,T(\QQ_3)^0.$$
\end{prop}

\pagebreak

\begin{proof}
    Reprenons le contexte de $\S$\ref{LeCorpsDeNombres}. Sous les notations du lemme \ref{CalculCokerMu}, on a $d_\Sigma=4$ et $d_{\overline{\Sigma}}=2$. Posons $\mu\colonequals R_{K'/K}(\mu_4)/\mu_4$. Ce même lemme donne alors
    $$\coker \left(H^1(K,\mu)\rightarrow H^1(K_\Sigma,\mu)\right)\overset{\sim}{\rightarrow} \ZZ/\frac{\pgcd(d_\Sigma,4)}{\pgcd(d_\Sigma,d_{\overline{\Sigma}},4)} \ZZ = \ZZ/2\ZZ.$$
    
    On a donc le résultat en utilisant la proposition \ref{stratégie}.
\end{proof}

\subsection{Exemple dans le cas des corps de fonctions}

\begin{pp}[{Étude d'un corps de fonctions}]
Soit $p$ un premier tel que $p\equiv 3 \pmod 4$. Posons ${K\colonequals \FF_p(t)}$, $K'\colonequals \FF_p(\sqrt{t},\sqrt{1-t})$ et $\Sigma\colonequals \{t^{-1}\}$. On a donc $K_\Sigma=\FF_p(\!(t^{-1})\!)$. Pour ce qui est de $K'_{\Sigma'}$, on a :
$$K'_{\Sigma'}\colonequals K'\otimes_{\FF_p(t)} \FF_p(\!(t^{-1})\!) \cong \FF_p(\!(t^{-1})\!)[X,Y]/\left(X^2-t,Y^2-(1-t)\right).$$

Notons que $X^2-t$ est irréductible dans $\FF_p(\!(t^{-1})\!)$ car, dans le cas contraire, cela donnerait l'existence d'un élément de valuation $-\frac{1}{2}$, ce qui est absurde.
Observons ensuite que ${1-t=(-1)\,t\,(1-t^{-1})}$. On sait que $t$ est un carré dans $\FF_p(\!(t^{-\frac{1}{2}})\!)$. Aussi, $1-t^{-1}$ est également un carré dans $\FF_p(\!(t^{-\frac{1}{2}})\!)$ grâce au lemme de Hensel, puisque l'est sur $\FF_p$ en réduisant modulo $t^{-\frac{1}{2}}$. \linebreak En conséquence, $X^2-(1-t)$ est irréductible dans $\FF_p(\!(t^{-\frac{1}{2}})\!)$ si et seulement si $-1$ n'est pas un carré. C'est bien le cas puisque la condition $p\equiv 3 \pmod 4$ implique que $-1$ n'est pas un carré dans $\FF_p$.
On en déduit donc que $\Sigma'=\{t^{-\frac{1}{2}}\}$, et que $K'_{\Sigma'}$ est un corps isomorphe à $\FF_{p^2}(\!(t^{-\frac{1}{2}})\!)$. \linebreak C'est une extension modérée de $\FF_p(\!(t^{-1})\!)$ de groupe de Galois $\ZZ/2\ZZ\times \ZZ/2\ZZ$.
\medskip

Notons que $K$ ne se ramifie potentiellement que en $t$, $1-t$ et $t^{-1}$. Les autres places sont alors non ramifiées, donc de groupe de Galois local cyclique, et donc de degré au plus $2$. En effet, en une autre place, $t$ et $1-t$ sont des unités, de sorte que leur réduction dans $\FF_p$ soit non triviale. En conséquence, les réductions dans $\FF_p[X]$ des polynômes $X^2-t$ et $X^2-(1-t)$ sont séparables ($p\not = 2$). Ils admettent donc chacun une racine dans $\FF_p^\mathrm{sep}$ et donc sur l'extension maximale non ramifiée d'après Hensel. D'où la non-ramification des extensions locales.

\begin{itemize}[leftmargin=0.5cm]
    \item En $t$, on a bien sûr $X^2-t$ qui est irréductible puisque $t$ est une uniformisante. D'autre part, $X^2-(1-t)$ admet une racine dans $\FF_p$, en réduisant modulo $t$. Le lemme de Hensel affirme alors qu'il y en a une dans $\FF_p(\!(t)\!)$. On obtient alors :
    $$K'\otimes_{\FF_p} \FF_p(\!(t)\!) \cong \FF_p(\!(t)\!)[X,Y]/\left(X^2-t,Y^2-(1-t)\right) \cong \FF_p(\!(t^{\frac{1}{2}})\!)[X]/\left(X^2-(\sqrt{1-t})^2\right) \cong \left(\FF_p(\!(t^{\frac{1}{2}})\!)\right)^2.$$
    Les deux corps locaux de $K'$ au-dessus de $t$ sont donc des extensions cycliques de degré $2$.
    \medskip

    \item En $1-t$, le polynôme $X^2-(1-t)$ est irréductible puisque $1-t$ est une uniformisante. Aussi, puisque $t = 1 - (1-t)$, le polynôme $X^2-t$ admet une racine dans $\FF_p$, en réduisant modulo $1-t$. Le lemme de Hensel en donne donc une dans $\FF_p(\!(1-t)\!)$. On a alors :
    \[\begin{adjustbox}{max size={1\textwidth}{1\textheight}}
    $K'\otimes_{\FF_p} \FF_p(\!(1-t)\!) \cong \FF_p(\!(1-t)\!)[X,Y]/\left(X^2-t,Y^2-(1-t)\right) \cong \FF_p\left(\!\left((1-t)^{\frac{1}{2}} \right)\!\right)[X]/\left(X^2-(\sqrt{t})^2\right) \cong \FF_p\left(\!\left((1-t)^{\frac{1}{2}}\right)\!\right)^2.$
    \end{adjustbox} \]
    Ainsi, les deux corps locaux de $K'$ au-dessus de $1-t$  sont des extensions cycliques de degré $2$.
\end{itemize}
\end{pp}

\begin{prop}\label{ContreExempleCorpsDeFonction}
    Soit $p$ un premier tel que $p\equiv 3 \pmod 4$. 
    Il existe un $\FF_p(t)$-groupe semi-simple $G$ et un $\FF_p(t)$-tore $T$ tel que $T\left(\FF_p(\!(t^{-1})\!)\right)^0=T\left(\FF_p(\!(t^{-1})\!)\right)^1$, tous deux déployés sur $\FF_p(\sqrt{t},\sqrt{1-t})$, tels que :
    $$0\not = \ZZ/2\ZZ \overset{\sim}{\leftarrow} G\left(\FF_p(\!(t^{-1})\!)\right)/\overline{G(\FF_p(t))}\overset{\sim}{\rightarrow}G\left(\FF_p(\!(t^{-1})\!)\right)/G(\FF_p(t))\,G\left(\FF_p(\!(t^{-1})\!)\right)^0 \text{\hspace{2pt} et }$$
    $$0\not = \ZZ/2\ZZ \overset{\sim}{\leftarrow} T\left(\FF_p(\!(t^{-1})\!)\right)/\overline{T(\FF_p(t))}\overset{\sim}{\rightarrow}T\left(\FF_p(\!(t^{-1})\!)\right)/T(\FF_p(t))\,T\left(\FF_p(\!(t^{-1})\!)\right)^0.$$
\end{prop}

\begin{proof}
    La démonstration est identique à celle de la proposition \ref{ContreExempleCorpsDeNombre}.
\end{proof}

\pagebreak

\section{Contre-exemples à la surjectivité du morphisme de Kottwitz}\label{SectionCEXSurjectivite}

On considère ici $K=\RR(t)$ et $\Sigma=\{t\}$, de sorte que $K_\Sigma=\RR(\!(t)\!)$, que l'on note désormais $\Kh$. \linebreak On construit dans cette partie deux exemples de tores $T$ sur $K$ qui vérifient tous deux l'égalité $T(\Kh)=T(K)\,T(\Kh)^0$, mais pour lesquels le morphisme de Kottwitz $\kappa_{T_{\Kh}}$ n'est pas surjectif \linebreak (cf. les propositions \ref{1erCEX} et \ref{2eCEX}).

\begin{pp}[{Contexte des contre-exemples}]
Posons $L\colonequals \CC(u)$, $\Lh\colonequals \CC(\!(u)\!)$, $\Oc_{\Lh}\colonequals  \CC[[u]]$, et $\Oc_{\Kh}\colonequals  \RR[[t]]$, où $u$ est tel que $u^2=t$. Dénotons par $v_{\Lh}$, la valuation de $\Lh$. Considérons ${\Gamma\colonequals \Gal(L/K)\cong\Gal(\Lh/\Kh)\cong\ZZ/2\ZZ \times \ZZ/2\ZZ}$ engendré par $\sigma:i\mapsto -i$ et $\tau:u\mapsto -u$. Notons que $\Khnr = \CC(\!(t)\!)$. Dans la suite, il est inoffensif de considérer que $I$ désigne $\Gal(\Lh/\Khnr)$.
\medskip

Soit $T$ un $\Kh$-tore déployé sur $\Lh$. Rappelons que l'on a une identification de $\Gamma$-modules ${X\otimes L^\times \overset{\sim}{\rightarrow} T(\Lh)}$ donnée par $\chi\otimes l \mapsto \chi(l)$. De même, $X\otimes \Oc_{\Lh}^\times \overset{\sim}{\rightarrow} T(\Lh)^1$. Notons d'ailleurs que la décomposition $\Oc_{\Lh}^\times\cong \CC^\times \times U^1$ (où $U^1$ est le sous-groupe de $\Oc_{\Lh}^\times$ dont le terme constant est $1$) est $\Gamma$-invariante et induit alors la décomposition $\Gamma$-invariante $X\otimes \Oc_{\Lh}^\times \overset{\sim}{\rightarrow} (X\otimes \CC^\times) \oplus (X\otimes U^1)$.
\medskip

Rappelons d'après \cite[Proposition 11.1.1]{kaletha_prasad_2023} que l'on a le diagramme suivant :
\begin{equation}
\begin{tikzcd}
	{T(\Lh)} & X \\
	{T(\Khnr)} & {X_I}
	\arrow["{\id\otimes v_{\Lh}}", from=1-1, to=1-2]
	\arrow["{N_{\Lh/\Khnr}}"'{pos=0.4}, two heads, from=1-1, 
to=2-1]
	\arrow["{\overline{\cdot}}", two heads, from=1-2, to=2-2]
	\arrow["{\kappa_{T_{\Khnr}}}"', two heads, from=2-1, 
to=2-2]
\end{tikzcd}
\end{equation}
où $N$ est la norme et $\overline{\cdot}\colonequals X\rightarrow X_I$ la projection canonique. En conséquence, $N(T(\Lh)^1)=T(\Khnr)^0$. Ceci donne alors la décomposition $T(\Khnr)^0 = N(X\otimes \CC^\times) \oplus N(X\otimes U^1)$. Aussi, puisque $\CC^\times$ est un groupe divisible $\tau$-invariant, on a $N(X\otimes \CC^\times)=N(X)\otimes \CC^\times = X^I \otimes \CC^\times$. Et enfin, puisque $U^1$ a une structure de $\QQ$-espace vectoriel, $\widehat{H}^0(I,X\otimes U^1)=0$ et donc $N(X\otimes U^1)=(X\otimes U^1)^I$. D'où finalement :
\begin{equation}
T(\Khnr)^0\cong(X^I \otimes \CC^\times) \oplus (X\otimes U^1)^I
\end{equation}
Notons d'ailleurs que cette décomposition est $\Galnr$-invariante.
\end{pp}

\subsection{Premier contre-exemple}\label{Section1erCEX}

\begin{pp}[{Le module galoisien du tore normique}]\label{Contexte1erCEX}
Considérons le tore normique $T\colonequals R^{(1)}_{L/K}(\GG_m)$. Son module des cocaractères $X$ est le sous-réseau de $\ZZ[\Gamma]$ de trace nulle. Autrement dit, si l'on note $(e_1,e_\tau,e_\sigma,e_{\sigma \tau})$ la base naturelle de $\ZZ[\Gamma]$, on a:
$$X\colonequals \{ \sum_{\alpha \in \Gamma} a_\alpha e_\alpha \in \ZZ[\Gamma] \mid \sum_{\alpha \in \Gamma} a_\alpha = 0 \}.$$
On vérifie qu'une base de $X$ est donnée par $(a,b,c)\colonequals (e_1-e_\tau,e_\tau-e_\sigma,e_\sigma - e_{\sigma\tau})$. Les actions de $\sigma$ et $\tau$ sur la base sont alors données par les matrices suivantes:
$$\text{pour } \sigma:\begin{pmatrix}
0 & -1 & 1\\
0 & -1 & 0\\
1 & -1 & 0
\end{pmatrix} \text{ et pour } 
\tau:\begin{pmatrix}
-1 & 1 & 0\\
0 & 1 & 0\\
0 & 1 & -1
\end{pmatrix}.
$$

Calculons alors $X_I$. On a par définition :
\begin{align*}
    \overline{a}&=\tau(\overline{a})=\overline{\tau(a)}=-\overline{a}\\
    \overline{b}&=\tau(\overline{b})=\overline{\tau(b)}=\overline{a+b+c}\\
    \overline{c}&=\tau(\overline{c})=\overline{\tau(c)}=-\overline{c}.
\end{align*}
D'où les relations $2\overline{a}=0$ et $\overline{a}=\overline{c}$. Par conséquent, $X_I\cong \ZZ/2\ZZ \times \ZZ$ avec comme base $(\overline{a},\overline{b})$. \linebreak Par ailleurs, l'action de $\sigma$ sur $X_I$ est alors donnée par les relations :
\begin{align*}
    \sigma(\overline{a})&=\overline{\sigma(a)}=\overline{c}=\overline{a}\\
    \sigma(\overline{b})&=\overline{\sigma(b)}=\overline{-(a+b+c)}=-\overline{b}.
\end{align*}
On en déduit alors que $(X_I)^{\Galnr}\cong\ZZ/2\ZZ$ avec comme générateur $\overline{a}$.
\medskip

Calculons cette fois $X^I$. Prenons $n_1a+n_2b+n_3c \in X^I$. On a :
$$n_1a+n_2b+n_3c=\tau(n_1a+n_2b+n_3c)=-n_1a + n_2(a+b+c)-n_3c=(n_2-n_1)a+n_2b+(n_2-n_3)c.$$
On en déduit $2n_1=2n_3=n_2$. Donc $X^I\cong \ZZ$ et un générateur est donné par $E\colonequals a+2b+c$. \linebreak Par ailleurs, $\sigma(E)=\sigma(a)+2\sigma(b)+\sigma(c)=c-2(a+b+c)+a=-E$. D'où l'action de $\sigma$ sur $X^I$.
\end{pp}

\begin{prop}\label{1erCEX}
    Le $K$-tore normique $T\colonequals R^{(1)}_{L/K}(\GG_m)$ est tel que 
    $$0\cong\Ima(\kappa_{T_{\Kh}}) \subsetneq ((T_*)_I)^{\Galnr}\cong \ZZ / 2\ZZ.$$
    En particulier, $T(\Kh)=T(\Kh)^0=T(K)\,T(\Kh)^0$. 
\end{prop}

\begin{proof}
Reprenons le contexte du $\S$\ref{Contexte1erCEX}. Posons $l = a(u)$ et calculons $N(l)$. On trouve :
$$N(l)= l \tau(l) = a(u)\left(\tau(a)(\tau(u))\right)=a(u)a^{-1}(-u)=a(u)a((-u)^{-1})=a(-1).$$
Par ailleurs, la commutativité du diagramme $(1)$ montre que ${\kappa_{T_{\Khnr}}}(a(-1))=\overline{a} \in (X_I)^{\Galnr}$, puisque $v_{\Lh}(u)=1$. Toutefois, $\sigma(a(-1))=\sigma(a)(\sigma(-1))=c(-1)$. En conséquence, $a(-1)\not \in T(\Kh)$. Montrons que $\overline{a}$ n'est pas dans l'image de $\kappa_{T_{\Kh}}$.

Considérons le cocycle $z:\alpha \mapsto a(-1) \alpha(a(-1))^{-1}$ à valeurs dans $T(\Khnr)^0$ (puisque ${\sigma(\overline{a})=\overline{a}}$). \linebreak Ce cocycle est un cobord si et seulement si $\overline{a}\in \Ima(\kappa_{T_{\Kh}})$. Notons d'ailleurs que :
$$a(-1) \sigma(a(-1))^{-1}=a(-1)c(-1)^{-1}=a(-1)c(-1)=a(-1)b(-1)^2c(-1)=E(-1) \in X^I \otimes \CC^\times.$$

Par l'absurde, supposons que $z$ soit un cobord. Il s'écrit donc $\alpha \mapsto y \alpha(y)^{-1}$ avec $y\in T(\Khnr)^0$. Utilisons la décomposition $(2)$ pour écrire $y=y_0u$. On a alors $y \sigma(y)^{-1}=(y_0 \sigma(y_0)^{-1}) (u \sigma(u)^{-1})$.

En identifiant les décompositions, on a $E(-1)=y_0 \sigma(y_0)^{-1}$. Par ailleurs, puisque $X^I = \langle E \rangle$, on peut écrire $y_0=E(x)$ avec $x\in \CC^\times$. D'où finalement :
$$E(-1)=y_0 \sigma(y_0)^{-1}=E(x) \sigma(E)(\sigma(x)^{-1})=E(x)E(\sigma(x))=E(|x|^2).$$
On en déduit alors $|x|^2 = -1$. C'est absurde ! Donc $z$ n'est pas un cobord et $\overline{a}\not \in \Ima(\kappa_{T_{\Kh}})$.

Puisque $\ZZ/2\ZZ \cong (X_I)^{\Galnr}=\langle\overline{a}\rangle$, on en déduit que $\Ima(\kappa_{T_{\Kh}})=0$, et donc que $T(\Kh)=T(\Kh)^0$. \linebreak Il en découle alors trivialement la décomposition $T(\Kh)=T(K)\,T(\Kh)^0$. 
\end{proof}

\subsection{Second contre-exemple}

\begin{pp}[{Le module galoisien d'un tore induit sur $\CC(t)$}]\label{Contexte2eCEX}
Considérons cette fois le réseau $X=\ZZ^4$ dont la base canonique est notée $(e,e',f,f')$. On associe à $\sigma$ et $\tau$ les matrices suivantes :
$$\text{pour } \sigma:\begin{pmatrix}
A & 0 \\
I_2-A & -A 
\end{pmatrix} \text{ et pour } 
\tau:\begin{pmatrix}
A & 0 \\
0 & A 
\end{pmatrix}
\text{, où }
A=
\begin{pmatrix}
0 & 1 \\
1 & 0 
\end{pmatrix} 
.$$
Un calcul par blocs immédiat montre que cela définit une structure de $\Gamma$-module sur $X$. \linebreak On en déduit alors un unique $K$-tore $T$ tel que $T_*=X$. Par construction, $T_{\CC(t)}$ est induit.
\medskip

Calculons $X_I$. Puisque $\tau(e)=e'$ et $\tau(f)=f'$, on trouve que $X_I$ est défini par les relations $\overline{e}=\overline{e'}$ et $\overline{f}=\overline{f'}$, de sorte que $X_I \cong \ZZ^2$ avec comme base $(\overline{e},\overline{f})$. Sa $\Galnr$-action est donnée par :
\begin{align*}
    \sigma(\overline{e})&=\overline{\sigma(e)}=\overline{e'+f-f'}=\overline{e}\\
    \sigma(\overline{f})&=\overline{\sigma(f)}=\overline{-f'}=-\overline{f}.
\end{align*}
On en déduit alors que $(X_I)^{\Galnr}\cong\ZZ$ avec comme générateur $\overline{e}$.
\medskip

Calculons maintenant $X^I$. Prenons $ne+n'e'+mf+m'f'\in X^I$. On a :
$$ne+n'e'+mf+m'f'=\tau(ne+n'e'+mf+m'f')=n'e+ne'+m'f+mf'.$$
D'où $n=n'$ et $m=m'$. Donc $ne+n'e'+mf+m'f'=n(e+e')+m(f+f')$. On en déduit alors que $X^I\cong \ZZ^2$ et admet une base donnée par $(e+e',f+f')$. L'action de $\Galnr$ sur $X^I$ est donnée par :
\begin{align*}
    \sigma(e+e')&=\sigma(e)+\sigma(e')=(e'+f-f')+(e-f+f')=e+e'\\
    \sigma(f+f')&=\sigma(f)+\sigma(f')=-f'-f=-(f+f').
\end{align*}
\end{pp}

\begin{prop}\label{2eCEX}
    Le $K$-tore $T$ induit sur $\CC(t)$ introduit en $\S$\ref{Contexte2eCEX} est tel que
    $$2\ZZ\cong\Ima(\kappa_{T_{\Kh}}) \subsetneq ((T_*)_I)^{\Galnr}\cong \ZZ.$$
    Par ailleurs, $T(\Kh)=T(K)\,T(\Kh)^0$ et $T(\Khnr)^0=T(\Khnr)^1$ (i.e., $(T_*)_I$ est libre).
\end{prop}

\begin{proof}
Posons $l = e(iu)$ et calculons $N(l)$. On trouve :
$$N(l)= l \tau(l) = e(iu)\left(\tau(e)(\tau(iu))\right)=e(iu)e'(-iu)=(e+e')(iu)e'(-1).$$
Par ailleurs, la commutativité du diagramme $(1)$ montre que ${\kappa_{T_{\Khnr}}}(N(l))=\overline{e} \in (X_I)^{\Galnr}$, puisque $v_{\Lh}(iu)=1$. Montrons que $\overline{e}$ n'est pas dans l'image de $\kappa_{T_{\Kh}}$.

Considérons le cocycle $z:\alpha \mapsto N(l) \alpha(N(l))^{-1}$ à valeurs dans $T(\Khnr)^0$ (puisque ${\sigma(\overline{e})=\overline{e}}$). \linebreak Ce cocycle est un cobord si et seulement si $\overline{e}\in \Ima(\kappa_{T_{\Kh}})$. Notons d'ailleurs que :
$$\sigma(N(l))=\sigma((e+e')(iu)e'(-1))=\sigma(e+e')(-iu)\sigma(e')(-1)=(e+e')(-iu)(e-f+f')(-1).$$
Par conséquent :
\begin{align*}
    N(l)\sigma(N(l))^{-1}&=(e+e')(iu)e'(-1)\left((e+e')(-iu)(e-f+f')(-1)\right)^{-1} \\
    &=e'(-1)(e+e')(-1)(e-f+f')(-1) = (2e+2e'-f+f')(-1) = (f+f')(-1)\in X^I \otimes \CC^\times.
\end{align*}

Par l'absurde, supposons que $z$ soit un cobord. Il s'écrit donc $\alpha \mapsto y \alpha(y)^{-1}$ avec $y\in T(\Khnr)^0$. Utilisons la décomposition $(2)$ pour écrire $y=y_0u$. On a alors $y \sigma(y)^{-1}=(y_0 \sigma(y_0)^{-1}) (u \sigma(u)^{-1})$.

En identifiant les décompositions, on a $(f+f')(-1)=y_0\sigma(y_0)^{-1}$. Par ailleurs, comme $\ZZ(e+e')$ et $\ZZ(f+f')$ sont stables par $\sigma$, et que $(e+e',f+f')$ est une base de $X^I$, on a $y_0\in \ZZ (f+f') \otimes \CC^\times$. On l'écrit alors $y_0=(f+f')(x)$ pour $x\in \CC^\times$. D'où finalement :
$$(f+f')(-1)=y_0 \sigma(y_0)^{-1}=(f+f')(x) \sigma(f+f')(\sigma(x)^{-1})=(f+f')(x)(f+f')(\sigma(x))=(f+f')(|x|^2).$$
On en déduit alors $|x|^2 = -1$. C'est absurde ! Donc $z$ n'est pas un cobord et $\overline{e}\not \in \Ima(\kappa_{T_{\Kh}})$.

Toutefois $2 \overline{e} \in \Ima(T(K)\rightarrow (X_I)^{\Galnr})$. En effet $N(e(t))=(e+e')(t)\in T(K)$ et $v_{\Lh}(t)=2$. Il en découle que $\Ima(T(K)\rightarrow (X_I)^{\Galnr})=\Ima(\kappa_{T_{\Kh}})$, c'est à dire $T(\Kh)=T(K)\,T(\Kh)^0$.
\end{proof}

\begin{rema}
    Soit $T$ un $\Kh$-tore et $\mathcal{T}$ son modèle de Néron. On a :
    \begin{align*}
        \coker(\kappa_T)&\overset{\sim}{\to}\Ker\left(H^1(\Galnr,T(\Khnr)^0)\to H^1(\Galnr,T(\Khnr)\right)\\
        {} & \overset{\sim}{\to}\Ker\left(H^1(\Galnr,T(\Khnr)^0)\to H^1(\Gal(\Kh^\mathrm{sep}/\Kh),T(\Kh^\mathrm{sep})\right)\overset{\sim}{\to}\Ker\left(H^1(\Oc_{\Kh},\mathcal{T}^0)\to H^1(\Kh,T)\right).
    \end{align*}
    
    Le premier isomorphisme provient de la suite exacte en cohomologie induite par $\kappa_{T_{\Khnr}}$, le second provient de la suite exacte inflation-restriction \cite[I.\S5.8.a)]{CohoGalois}, et le dernier s'obtient grâce à \cite[2.9.2.(2)]{GilleSemisimpleSchemas} et \cite[XXIV, Proposition 8.1.(i)]{SGA3}.

    En conséquence, nos deux contre-exemples précédents induisent également un contre-exemple à un problème de type Grothendieck-Serre pour les modèles de Néron de tores.
\end{rema}

\section{Propriétés birationnelles}\label{SectionBirat}

Dans cette partie, $K$ n'est plus nécessairement un corps global. On prouve le théorème suivant :

\begin{theo}
    Soit $T$ un $K$-tore, et $1\rightarrow S \rightarrow P \rightarrow T \rightarrow 1$, une résolution flasque de $T$. Notons $(S_*)_{I,\mathrm{free}}$, le quotient de $(S_*)_{I}$ par sa partie de torsion. On a les isomorphismes naturels suivants :
    $$\coker(\kappa_{T_{K_\Sigma}}) \overset{\sim}{\rightarrow} \coker \left(H^1(K_\Sigma,S)\rightarrow H^1(\Gamma^\mathrm{nr},(S_*)_{I,\mathrm{free}}) \right),$$
    $$\coker(T(K)\rightarrow ((T_*)_I)^{\Galnr}) \overset{\sim}{\rightarrow} \coker \left(H^1(K,S)\rightarrow H^1(\Gamma^\mathrm{nr},(S_*)_{I,\mathrm{free}}) \right),\text{\hspace{2pt} et }$$
    $$T(K_\Sigma)/T(K)\, T(K_\Sigma)^0 \overset{\sim}{\rightarrow} \coker \left(H^1(K,S)\rightarrow \Ima(H^1(K_\Sigma,S)\rightarrow H^1(\Gamma^\mathrm{nr},(S_*)_{I,\mathrm{free}}) \right).$$
    En particulier, ces trois quantités sont des invariants birationnels stables de $T$.
\end{theo}

\begin{proof}
    Observons que l'on a la suite exacte :
    \[\begin{tikzcd}
    	H_1(I,T_*) & {(S_*)_I} & {(P_*)_I} & {(T_*)_I} & 1.
    	\arrow[from=1-1, to=1-2]
    	\arrow[from=1-4, to=1-5]
    	\arrow[from=1-2, to=1-3]
    	\arrow[from=1-3, to=1-4]
    \end{tikzcd}\]
    Comme $(P_*)_I$ est libre et $H_1(I,T_*)$ est de torsion, il vient immédiatement :
    \[\begin{tikzcd}
    	1 & {(S_*)_{I,\mathrm{free}}} & {(P_*)_I} & {(T_*)_I} & 1.
    	\arrow[from=1-1, to=1-2]
    	\arrow[from=1-4, to=1-5]
    	\arrow[from=1-2, to=1-3]
    	\arrow[from=1-3, to=1-4]
    \end{tikzcd}\]
    Notons que $(P_*)_I$ est un module de permutation puisque $P_*$ l'est. Par ailleurs, $\kappa_{P_{K_\Sigma}}$ est surjectif, car cela est vrai pour $\GG_m$. On obtient alors le diagramme suivant à lignes exactes :
    \[\begin{tikzcd}
    	 {...} & {P(K_\Sigma) } & {T(K_\Sigma) } & 
    {H^1(K_\Sigma,S)} & {H^1(K_\Sigma,P)=1} \\
    	 {...} & 
    {((P_*)_I)^{\Gamma^{\mathrm{nr}}}} & 
    {((T_*)_I)^{\Gamma^{\mathrm{nr}}}} & 
    {H^1(\Gamma^{\mathrm{nr}},(S_*)_{I,\mathrm{free}})} & 
    {H^1(\Gamma^{\mathrm{nr}},(P_*)_{I})=1}
    	\arrow[from=1-1, to=1-2]
    	\arrow[from=1-2, to=1-3]
    	\arrow[from=1-2, to=2-2]
    	\arrow[from=1-3, to=1-4]
    	\arrow["{\kappa_{P_{K_\Sigma}}}", two heads, from=1-2, to=2-2]
    	\arrow[from=1-4, to=1-5]
    	\arrow["{\kappa_{T_{K_\Sigma}}}", from=1-3, to=2-3]
    	\arrow[from=1-4, to=2-4]
    	\arrow[from=2-1, to=2-2]
    	\arrow[from=2-2, to=2-3]
    	\arrow[from=2-3, to=2-4]
    	\arrow[from=2-4, to=2-5]
    \end{tikzcd}\]
    Une chasse aux diagrammes donne alors le premier isomorphisme.
    \medskip

    Par ailleurs, $P(K)\rightarrow (P_*)_{I}$ est également surjectif car $P(K)$ est dense dans $P(K_\Sigma)$. En remplaçant dans la première ligne du diagramme précédent $K_\Sigma$ par $K$, la même chasse aux diagrammes donne alors le second isomorphisme.
    \medskip

    Le dernier isomorphisme s'obtient en observant que 
    \begin{align*}
        T(K_\Sigma)/T(K)\,T(K_\Sigma)^0 & \overset{\sim}{\rightarrow}(T(K_\Sigma)/T(K_\Sigma)^0)/(T(K)/T(K)\cap T(K_\Sigma)^0)\\
        &\overset{\sim}{\rightarrow} \Ima(\kappa_{T_{K_\Sigma}})/\Ima(T(K)\rightarrow (T_*)_I)\\
        &\overset{\sim}{\rightarrow}\coker(T(K)\rightarrow (T_*)_I)/\coker(\kappa_{T_{K_\Sigma}})\\
        &\overset{\sim}{\rightarrow}  \coker \left(H^1(K,S)\rightarrow H^1(\Gamma^\mathrm{nr},(S_*)_{I,\mathrm{free}}) \right)/\coker \left(H^1(K_\Sigma,S)\rightarrow H^1(\Gamma^\mathrm{nr},(S_*)_{I,\mathrm{free}}) \right)\\
        &\overset{\sim}{\rightarrow}\coker \left(H^1(K,S)\rightarrow \Ima(H^1(K_\Sigma,S)\rightarrow H^1(\Gamma^\mathrm{nr},(S_*)_{I,\mathrm{free}}) \right).
    \end{align*}
    
    Enfin, observons que ces trois quantités sont invariantes si l'on remplace $S$ par $S\times P_0$ où $P_0$ est un tore quasi-trivial. On conclut alors que ces quantités sont des invariants birationnels stables de $T$ d'après \cite[\S 4]{VoskresenskiiBook}.
\end{proof}

\section*{Remerciements}

L'auteur remercie Philippe Gille et Ralf Köhl, pour leurs soutiens, leurs accompagnements et leur relecture de la présente note.

L'auteur remercie également la Studienstiftung des deutschen Volkes pour avoir supporté financièrement cette note. Il a été également soutenu par le projet “Group schemes, root systems, and related representations”, financé par l'Union Européenne - NextGenerationEU à travers le "Romania’s National Recovery and Resilience Plan" (PNRR) call no. PNRR-III-C9-2023-I8, Project CF159/31.07.2023, et coordonné par le Ministre de la Recherche, de l'Innovation et de la Numérisation (MCID) de Roumanie.

\printbibliography